\documentclass{IEEEtran} 
\IEEEoverridecommandlockouts                              % This command is only
\usepackage[margin=0.30cm]{caption}
\usepackage{epsf}
\usepackage{epsfig}
\usepackage{times}
\usepackage{amssymb}
\usepackage{amsmath}
\usepackage{amsfonts}
\usepackage{latexsym}
\usepackage{url}
\usepackage{mathrsfs}
\usepackage{algorithm}
\usepackage{caption}
\usepackage{datetime}
\usepackage{bm}
\usepackage{algorithm,algorithmic}
\usepackage{subfigure}
\usepackage{framed} 
\usepackage[framed]{ntheorem}
\usepackage{cite}

\usepackage{lipsum}

\usepackage{amstext}  
\usepackage{array}      

\usepackage{mathabx}
\usepackage{sansmath}

\usepackage{mathtools}

\usepackage{color}

\newframedtheorem{frm-definition}{Definition}

\newtheorem{theorem}{Theorem}
\newtheorem{definition}{Definition}
\newtheorem{lemma}{Lemma}

\newtheorem{assumption}{Assumption}

\newcommand{\hL}{\mathcal{L}}
\newcommand{\hN}{\mathcal{N}}

\newcommand{\cL}{{\cal{L}}}
\newcommand{\cN}{{\cal N}}

\newcommand{\cT}{{\cal T}}

\newcommand{\cE}{{\cal E}}

\newcommand{\cK}{{\cal K}}

\newcommand{\cU}{{\cal U}}

\newcommand{\cY}{{\cal Y}}

\begin{document}
\title{\LARGE Accelerated Voltage Regulation in Multi-Phase Distribution Networks\\Based on Hierarchical Distributed Algorithm}
\author{Xinyang Zhou, Zhiyuan Liu, Changhong Zhao, Lijun Chen
\thanks{X. Zhou is with the National Renewable Energy Laboratory, Golden, CO, 80401, USA (Emails: xinyang.zhou@nrel.gov). }
\thanks{Z. Liu and L. Chen are with the College of Engineering and Applied Science, University of Colorado, Boulder, CO 80309, USA (Emails: \{zhiyuan.liu, lijun.chen\}@colorado.edu).}
\thanks{C. Zhao is with the Department of Information Engineering, the Chinese University of Hong Kong, HK (Emails: zhchangh1987@gmail.com).}
\thanks{Preliminary results of this paper have been accepted to American Control Conference, Philadelphia, 2019 \cite{zhou2018hierarchical}.}}
\maketitle

\begin{abstract}
We propose a hierarchical distributed algorithm to solve optimal power flow (OPF) problems that aim at dispatching controllable distributed energy resources (DERs) for voltage regulation at minimum cost. The proposed algorithm features unprecedented scalability to large multi-phase distribution networks by
jointly exploring the tree/subtrees structure of a large radial distribution network and the structure of the linearized distribution power flow (LinDistFlow) model to derive a hierarchical, distributed implementation of the primal-dual gradient algorithm that solves OPF. The proposed implementation significantly reduces the computation loads compared to the centrally coordinated implementation of the same primal-dual algorithm without compromising optimality. 
Numerical results on a 4,521-node test feeder show that the designed algorithm achieves more than 10-fold acceleration in the speed of convergence compared to the centrally coordinated primal-dual algorithm through reducing and distributing computational loads.
%{\color{cyan} Waiting for a centralized implementation for comparison.}
%, and that multi-phase algorithm achieves better optimality than the its single-phase counterpart. We also show that the proposed hierarchical distributed algorithm can achieve a ??-fold improvement in the speed of convergence compared to the centrally coordinated primal-dual algorithm. 
\end{abstract}

\section{Introduction}

%{\color{red}What is the problem?
%	Why is it interesting and important?
%	Why is it hard? (E.g., why do naive approaches fail?)
%	Why hasn't it been solved before? (Or, what's wrong with previous proposed solutions? How does mine differ?)
%	What are the key components of my approach and results? Also include any specific limitations.}

OPF problems determine the best operating points of dispatchable devices in electric power grids and achieve the optimal system-wide objectives and operational constraints for important applications such as demand response and voltage regulation. However, the increasing penetrations of DERs such as roof-top photovoltaic, electric vehicles, battery energy storage systems, thermostatically controlled loads, and other controllable loads not only provide enormous potential optimization and control flexibility that we can explore \cite{dall2017unlocking}, but also make OPF more challenging to solve with significantly growing dimensionality.
Meanwhile, due to the intermittent nature of the renewable energy resources, their deepening penetration in the distribution networks causes large and rapid fluctuations in power injections and voltages, and calls for fast control paradigms.

Distributed algorithms are developed to facilitate scalable and fast control of large networks of dispatchable DERs by parallel and distributed computation. In literature, such algorithms have been designed and implemented either with a central coordinator (CC), e.g., \cite{dall2018optimala, hauswirth2016projected, tang2017real, mohammadi2018agent, zhou2017incentive, zhou2017discrete}, or among neighboring agents without a CC, e.g., \cite{peng2018distributed, magnusson2017voltage, vsulc2014optimal, wu2018smart,gan2016online}. The latter usually demands that all nodes in the network compute and pass along updated information to their neighbors, which may not be implementable in practice if not all nodes are controllable or equipped with computation and communication capability. This paper will focus on the former communication model, which only requires CC to communicate with controllable nodes.
  However, what has been overlooked by existing literature on centrally coordinated distributed algorithms is that, while part of the computation loads are distributed among control nodes, the remaining that are executed by CC may surge as the size of the problem increases. Take voltage regulation problems (to be elaborated later) as an example: let $N$ denote the number of networked nodes, and the computational complexity of CC is in the order of $N^2$. When $N$ grows, the computational time of CC becomes exceedingly notable, rendering CC the bottleneck and hindering fast algorithm implementation in large systems. 

One way to mitigate the computational loads of CC (and therefore accelerate the algorithms) is to consider multiple coordinators among which CC's loads can be distributed. However, due to the complexity of the coupling term calculated by CC, it is usually challenging to decompose it in an efficient and yet structurally meaningful way.  To our best knowledge, little work has been done to explore this area. In particular, hierarchical structures have been under-exploited in developing distributed OPF algorithms.  

Moreover, realistic distribution systems are usually featured with multi-phase unbalanced loads.  The resultant multi-phase OPF has been studied extensively in literature, e.g., \cite{kekatos2016voltage, robbins2016optimal, dall2013distributed, gan2014convex, gan2016online, zhao2017convex}. Indeed, the inter-phase coupling in multi-phase systems adds extra difficulty to the already complex OPF problems and makes computational load reduction and distribution even more challenging. 

This work considers a large multi-phase distribution network with tree topology. We consider the linearized distribution flow (LinDistFlow) model \cite{baran1989optimala, baran1989optimalb,baran1989network, zhou2018reverse,liu2018signal} as well as its multi-phase extension \cite{gan2016online,gan2014convex} for  the distribution networks.
The network is divided into areas featuring subtree topology. A regional coordinator (RC) communicates with all the dispatchable nodes within each subtree, and CC communicates with all RCs. Each RC knows only the topology and line parameters of the subtree that it coordinates, and CC knows only the topology and line parameters  of the reduced network which treats each subtree as a node. Given such information availability, we explore the topological structure of the LinDistFlow model to derive a hierarchical, distributed implementation of the primal-dual gradient algorithm that solves an OPF problem.  
The OPF problem minimizes the total cost over all the controllable DERs and a cost associated with the total network load subject to voltage regulation constraints. The proposed implementation significantly reduces the computation burden of CC compared to the centrally coordinated implementation of the primal-dual algorithm, not only by distributing computational loads among RCs but also by reducing repetitive calculations and information transfers. Convergence of the designed algorithm is guaranteed with unbalanced nonlinear power flow.

Performance of the proposed implementation is verified through numerical simulation of a three-phase unbalanced 4,521-node test feeder with 1,043 controllable nodes. Simulation results show that a 10-fold acceleration in the speed of convergence can be achieved by the hierarchical distributed method compared to the centrally coordinated implementation. This significant improvement in convergence speed makes real-time grid optimization and control possible. Meanwhile, to our best knowledge, the size of the network in our simulation is the largest in distributed OPF studies.

It is worth noting that the model and algorithm design in this work can be readily applied to optimization and control of networked microgrids with each subtree seen as a microgrid. Meanwhile, most existing works on optimization and control of networked microgrids either over-simplify each individual microgrid as a node without considering power flow  within it \cite{wang2015coordinated,wang2016networked,fathi2013adaptive}, or ignore power flow models among microgrids \cite{utkarsh2018distributed,shi2015distributed,wang2016incentivizing}. Recent work \cite{zhang2018dynamic} applies a game-theoretic approach to manage a partitioned distribution network based on a noncooperative Nash game, where uniqueness of equilibrium, convergence, and global performance are difficult to characterize. Different from those in the literature, our method models inter- and intra-microgrid power flow dynamics and features provable convergence and global optimality performance.

Also note that, another line of works on voltage regulation are based on local control algorithms, e.g., \cite{zhu2015fast, singhal2018real, zhou2018reverse, liu2018signal} and references therein. Though local voltage control has the advantages of low computation and communication complexity, it can only solve a limited class of problems and cannot address more general objective functions and constraints as formulated in this work.

The rest of the paper is organized as follows. To provide design intuition, we start with Section~\ref{sec:model} to model the single-phase distribution system and formulate the OPF problem, and follow with Section~\ref{sec:hier} to propose a hierarchical distributed implementation of the primal-dual gradient algorithm. We then in Section~\ref{sec:multiphase} introduce the multi-phase distribution system and its OPF for which we elaborate a hierarchical distributed algorithm and characterize its convergence with nonlinear power flow. Section~\ref{sec:numerical} presents numerical results, and Section~\ref{sec:conclusion} concludes this paper.

\section{Solving OPF in Single-Phase System}\label{sec:model}
\subsection{Single-Phase Power Flow Model}
Consider a radial single-phase power distribution network denoted by $\cT=\{\cN \cup \{0\},\cE\}$ with $N+1$ nodes collected in the set $\cN \cup \{0\}$ where $\cN:=\{1, ..., N\}$ and node $0$ is the slack bus, and distribution lines collected in the set $\cE$. For each node $i\in \hN$, denote by $\cE_i \subseteq \cE$ the set of lines on the unique path from node $0$ to node $i$, and let $p_i$ and $q_i$ denote the real and reactive power injected, where negative (resp. positive) power injection means power consumption (resp. generation). Let $v_i$ be the squared magnitude of the complex voltage (phasor) at node $i$. For each line $(i, j)\in \cE$, denote by $r_{ij}$ and $x_{ij}$ its resistance and reactance, and  $P_{ij}$ and $Q_{ij}$ the real and reactive power from node $i$ to node $j$. Let $\ell_{ij}$ denote the squared magnitude of the complex branch current (phasor) from node $i$ to $j$. 

We adopt the following DistFlow model \cite{baran1989optimala, baran1989optimalb} for the radial distribution network:
\begin{subequations}\label{eq:bfm}
	\begin{eqnarray}
		\hspace{-5mm} P_{ij} \hspace{-2mm} &=&\hspace{-2mm} - p_j +\hspace{-2mm}\sum_{k: (j,k)\in \cE} \hspace{-2mm} P_{jk}+  r_{ij}  \ell_{ij}  \label{p_balance}, \\[-3pt]
		\hspace{-5mm}Q_{ij} \hspace{-2mm}&=&\hspace{-2mm}  -q_j + \hspace{-2mm}\sum_{k: (j,k)\in \cE} \hspace{-2mm} Q_{jk} + x_{ij} \ell_{ij} \label{q_balance},
	\end{eqnarray}
	\begin{eqnarray}
		\hspace{-5mm}v_j \hspace{-2mm}&=&\hspace{-2mm}  v_i - 2 \big(r_{ij} P_{ij} + x_{ij} Q_{ij} \big) + \big(r_{ij}^2+x_{ij}^2\big) \ell_{ij} \label{v_drop},\\[1pt]
		\hspace{-5mm}\ell_{ij}v_i \hspace{-2mm}&=&\hspace{-2mm}   P_{ij}^2 + Q_{ij}^2  \label{currents}.
	\end{eqnarray}
\end{subequations}

Following \cite{baran1989network, zhou2018reverse} we assume that the active and reactive power loss $r_{ij} \ell_{ij}$ and $ x_{ij} \ell_{ij}$, as well as $r^2_{ij} \ell_{ij}$ and $x^2_{ij} \ell_{ij}$, are negligible and can thus be ignored.   
Indeed, the losses are much smaller than power flows $P_{ij}$ and $Q_{ij}$, typically on the order of $1\%$. 
With the above approximations \eqref{eq:bfm} is simplified to the following linear model:
\begin{eqnarray}
\bm{v}&=&R\bm{p}+X\bm{q}+\tilde{\bm{v}},\label{eq:lindistflow}
\end{eqnarray}
where bold symbols $\bm{v}=[v_1,\ldots,v_N]^{\top}$, $\bm{p}=[p_1,\ldots,p_N]^{\top}$, $\bm{q}=[q_1,\ldots,q_N]^{\top}\in\mathbb{R}^N$ represent vectors, $\tilde{\bm{v}}$ is a constant vector with every component being the squared voltage magnitude at the slack bus, and the sensitivity matrices $R,X\in\mathbb{R}_+^{N\times N}$ respectively contain elements of
\begin{eqnarray}
	R_{ij}:= \!\!\! \sum_{(\zeta,\xi)\in \cE_i \cap \cE_j}\!\!\!\!2\cdot  r_{\zeta\xi}, \ \ \ \ X_{ij}:=\!\!\!\! \sum_{(\zeta,\xi)\in \cE_i \cap \cE_j}\!\!\!\! 2\cdot x_{\zeta\xi}  \label{X_def}.
\end{eqnarray}
Here, the voltage-to-power-injection sensitivity factors $R_{ij}$ (resp. $X_{ij}$) represents \textit{the resistance (resp. reactance) of the common path of node $i$ and $j$ leading back to node 0}. Keep in mind that this result serves as the basis for designing the hierarchical distributed algorithm to be introduced later. Fig.~\ref{fig:RX} (left) illustrates the common path $\cE_i \cap \cE_j$ for two arbitrary nodes $i$ and $j$ in a radial network and their $R_{ij},X_{ij}$.
\begin{figure}[h]
	\centering
	\includegraphics[width=.46\textwidth]{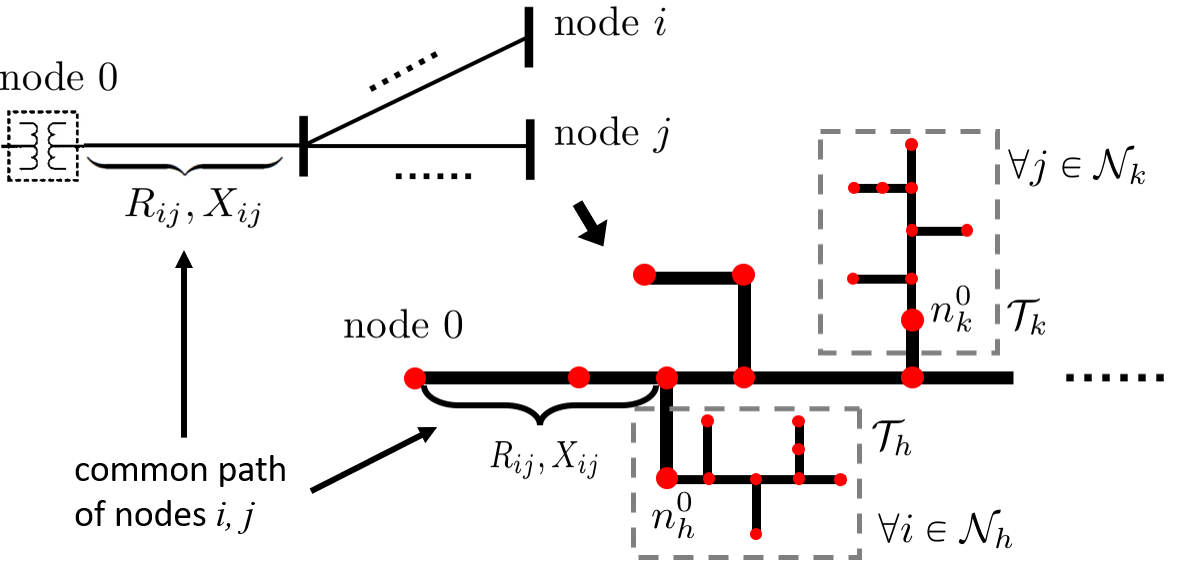}
	\caption{(Left) $\cE_i \cap \cE_j$ for  two arbitrary nodes $i, j$ in the network and their mutual voltage-to-power-injection sensitivity factors $R_{ij}, X_{ij}$. (Right) Unclustered nodes and root nodes of subtrees along with their connecting lines constitute the reduced network. Two subtrees $\cT_h$ and $\cT_k$ share the same common $R_{ij}$ and $X_{ij}$ for any of their respective nodes $i$ and $j$.}
	\label{fig:RX}
\end{figure}

\subsection{OPF and Primal-Dual Gradient Algorithm}
Assume node $i\in\cN$ has a dispatchable DER (or aggregation of DERs) whose real and reactive power injections are confined by $(p_i,q_i)\in\cY_i$ where $\cY_i$ is a convex and compact set. 
For example, the feasible region of a  PV inverter can be modeled as:
	\begin{align} 
	\cY_i =  \left\{ (p_i,q_i) \big|  0 \leq p_{i}  \leq  p_{i}^{\text{av}},~p_{i}^{2} + q_{i}^{2} \leq  \eta_{i}^2 \right\},\nonumber
	\end{align}
	where $p_{i}^{\text{av}}$ denotes the available active power from a PV system (based on prevailing ambient conditions), and  $\eta_{i}$ is the rated apparent capacity. The feasible set for an energy storage system can be modeled as 
	\begin{align} 
	\cY_i =  \left\{ (p_i,q_i) \big|  \underline{p}_i \leq p_{i}  \leq  \overline{p}_{i},~p_{i}^{2} + q_{i}^{2} \leq  \eta_{i}^2 \right\},\nonumber
	\end{align}
	for given limits $\underline{p}_i,\overline{p}_i$
	and for a given inverter capacity rating $\eta_i$. The limits $\underline{p}_i,\overline{p}_i$
	are updated during the operation of the
	battery based on the state of charge. The operating region of small-scale diesel generators can	be modeled using constant box constraints
	\begin{align} 
		\cY_i =  \left\{ (p_i,q_i) \big|  \underline{p}_i \leq p_{i}  \leq  \overline{p}_{i},\underline{q}_i \leq q_{i}  \leq  \overline{q}_{i}\right\}.\nonumber
	\end{align}

Let $P_0$ denote active power injected  into the feeder at node 0, which is linearly approximated by the total active power loads as
\begin{eqnarray}
 P_0=-P_I-\sum_{i\in\cN}p_i,\label{eq:P0}
\end{eqnarray}
where $P_I$ denotes the total uncontrollable (inelastic) power injection in the network. 
Use $\bm{v}(\bm{p},\bm{q})$ and $P_0(\bm{p})$ to represent \eqref{eq:lindistflow} and \eqref{eq:P0}, respectively, and consider the following OPF problem:
\begin{subequations}\label{eq:opt}
\begin{eqnarray}
&\underset{\bm{p},\bm{q}}{\min} & \sum_{i\in\cN}C_i(p_i,q_i)+C_0(P_0(\bm{p})),\\[-4pt]%+C_0(p,q)\\    
& \text{s.t.}&
\underline{\bm{v}}\leq \bm{v}(\bm{p},\bm{q}) \leq \overline{\bm{v}}, \label{eq:voltreg}\\
&     & (p_i,q_i)\in\cY_i,\forall i\in\cN,\label{eq:X}
\end{eqnarray}
\end{subequations}
where the objective $C_i(p_i,q_i)$ is the cost function for node $i$ and the coupling term $C_0(P_0(\bm{p}))$ represents the cost associated with the total network load. For example, $C_0(P_0(\bm{p}))=\alpha (P_0(\bm{p})-\tilde{P}_0)^2$ penalizes $P_0(\bm{p})$'s deviation from a dispatching signal $\tilde{P}_0$ from the bulk system operator, where $\alpha>0$ is a given weighting factor. We make the following assumption for these cost functions.
\begin{assumption}\label{ass:costfun}
$C_i(p_i,q_i),\: \forall i\in\cN$ are continuously differentiable and strongly convex in $(p_i,q_i)$, with bounded first-order derivative in $\cY_i$; $C_0(P_0)$ is continuously differentiable and convex with bounded first-order derivative.
\end{assumption}

Associate dual variables $\underline{\bm{\mu}}$ and $\overline{\bm{\mu}}$ with the left-hand-side and the right-hand-side of \eqref{eq:voltreg}, respectively, to write the Lagrangian of \eqref{eq:opt} as:
\begin{eqnarray}
    \hspace{-10mm}&&\cL(\bm{p},\bm{q};\overline{\bm{\mu}},\underline{\bm{\mu}})=\sum_{i\in\cN}C_i(p_i,q_i)+C_0(P_0(\bm{p}))\nonumber\\[-3pt]
    \hspace{-10mm}&&\hspace{5mm}+\underline{\bm{\mu}}^{\top}(\underline{\bm{v}}-\bm{v}(\bm{p},\bm{q}))+\overline{\bm{\mu}}^{\top}(\bm{v}(\bm{p},\bm{q})-\overline{\bm{v}}),\label{eq:lang}
\end{eqnarray}
with 
\eqref{eq:X} treated as the domain of $(\bm{p},\bm{q})$.

In order to design an algorithm with provable convergence, we introduce the following regularized Lagrangian with parameter $\eta>0$ and $\bm{\mu} :=[\underline{\bm{\mu}}^{\top}, \overline{\bm{\mu}}^{\top}]^{\top}$:
\begin{eqnarray}
     \hspace{-10mm}&&\cL_{\eta}(\bm{p},\bm{q};\overline{\bm{\mu}},\underline{\bm{\mu}})=\sum_{i\in\cN}C_i(p_i,q_i)+C_0(P_0(\bm{p}))\nonumber\\[-3pt]
     \hspace{-10mm}&&\hspace{5mm}+\underline{\bm{\mu}}^{\top}(\underline{\bm{v}}-\bm{v}(\bm{p},\bm{q}))+\overline{\bm{\mu}}^{\top}(\bm{v}(\bm{p},\bm{q})-\overline{\bm{v}})-\frac{\eta}{2}\|\bm{\mu}\|^2_2.\label{eq:langr}
\end{eqnarray}

 Since $\cL_{\eta}(\bm{p},\bm{q};\underline{\bm{\mu}},\overline{\bm{\mu}})$ is strongly convex in $\bm{p},\bm{q}$ and strongly concave in $\underline{\bm{\mu}},\overline{\bm{\mu}}$, the next result follows.
\begin{theorem}\label{the:unique}
There exists one unique saddle point $(\bm{p}^*,\bm{q}^*;\underline{\bm{\mu}}^*,\overline{\bm{\mu}}^*)$ of $\cL_{\eta}$.
\end{theorem}

Furthermore, the discrepancy due to the regularization term can be bounded and is proportional to $\eta$. We refer the details to \cite{koshal2011multiuser,zhou2017incentive}. For the rest of this section, we will focus on solving the saddle point of the regularized Lagrangian \eqref{eq:langr}.
Specifically, we cast the iterative projected primal-dual gradient algorithm to find the saddle point of \eqref{eq:langr} as follows:
\begin{subequations}\label{eq:primaldual}
\begin{eqnarray}
    \hspace{-2mm}\bm{p}(t+1)\hspace{-3mm}&=&\hspace{-3mm}\Big[\bm{p}(t)-\epsilon\Big(\!\nabla_{\bm{p}} C(\bm{p}(t),\!\bm{q}(t))\!-\!C'_0(P_0(t))\!\cdot\! \bm{1}_N\nonumber\\[-4pt]
    \hspace{-2mm}&&\hspace{22mm}+R^{\top}(\overline{\bm{\mu}}(t)-\underline{\bm{\mu}}(t))\Big)\Big]_{\bm{\cY}},\label{eq:primalP}\\[-4pt]
    \hspace{-2mm}\bm{q}(t+1)\hspace{-3mm}&=&\hspace{-3mm}\Big[\bm{q}(t)-\epsilon\Big(\!\nabla_{\bm{q}} C(\bm{p}(t),\bm{q}(t)) \nonumber\\[-4pt]
    \hspace{-2mm}&&\hspace{22mm}+X^{\top}(\overline{\bm{\mu}}(t)-\underline{\bm{\mu}}(t))\Big)\Big]_{\bm{\cY}},\label{eq:primalQ}\\[-4pt]
    \hspace{-2mm}\underline{\bm{\mu}}(t+1)\hspace{-3mm}&=&\hspace{-3mm}\big[\underline{\bm{\mu}}(t)+\epsilon (\underline{\bm{v}}-\bm{v}(t)-\eta   \underline{\bm{\mu}}(t))\big]_+,\label{eq:dual1}\\
    \hspace{-2mm}\overline{\bm{\mu}}(t+1)\hspace{-3mm}&=&\hspace{-3mm}\big[\overline{\bm{\mu}}(t)+\epsilon (\bm{v}(t)-\overline{\bm{v}}-  \eta\overline{\bm{\mu}}(t))\big]_+,\label{eq:dual2}\\
     \hspace{-2mm} \bm{v}(t+1)  \hspace{-3mm}&=&  \hspace{-3mm}R\bm{p}(t+1)+X\bm{q}(t+1)+\tilde{\bm{v}},\\
       \hspace{-2mm} P_0(t+1) \hspace{-3mm}&=&\hspace{-3mm}-P_I-\sum_{i\in\cN}p_i(t+1),
\end{eqnarray}
\end{subequations}
where $\epsilon>0$ is a constant stepsize to be determined, $\bm{1}_N=[1,\ldots,1]^{\top}\in\mathbb{R}^N$, $[\  ]_{\bm{\cY}}$ is the projection operator onto feasible sets $\bm{\cY}:=\bigtimes_{i\in\cN}\cY_i$, and $[ \ ]_{+}$ is the projection operator onto the positive orthant.

\subsection{Convergence Analysis}\label{sec:conv}
For uncluttered notation, we use $\bm{y}:=[\bm{p}^{\top},\bm{q}^{\top}]^{\top}$ to stack the primal variables and equivalently rewrite \eqref{eq:primaldual}  as follows:
\begin{eqnarray}
\begin{bmatrix}\bm{y}(t+1)\\ \bm{\mu}(t+1)\end{bmatrix}=\left[\begin{bmatrix}\bm{y}(t)\\ \bm{\mu}(t)\end{bmatrix}-\epsilon\begin{bmatrix}\nabla_{\bm{y}} \hL_{\eta}(\bm{y}(t), \bm{\mu}(t)) \\ -\nabla_{\bm{\mu}}\hL_{\eta}(\bm{y}(t),\bm{\mu}(t))
\end{bmatrix} \right]_{\bm{\cY}\times \bm{\cU}}\hspace{-4mm},\label{eq:mapk}
\end{eqnarray}
where $\bm{\cU}$ denotes the feasible positive orthant for the dual variables.
We further let $\bm{z}:=[\bm{y}^{\top},\bm{\mu}^{\top}]^{\top}$ stack all variables, and define the gradient operator $T(\bm{z}):=\begin{bmatrix}\nabla_{\bm{y}} \hL_{\eta}(\bm{y},\bm{\mu}) \\ -\nabla_{\bm{\mu}}\hL_{\eta}(\bm{y},\bm{\mu})\end{bmatrix}$.

\begin{lemma}\label{lem:mono}
$T(\bm{z})$ is a strongly monotone operator.
\end{lemma}
{We refer to Appendix for the proof.}

By Lemma~\ref{lem:mono}, there exists some constant $M>0$ such that for any $\bm{z},\bm{z'}\in\bm{\cY}\times \bm{\cU}$, one has
\begin{eqnarray}
(T(\bm{z})-T(\bm{z'}))^{\top}(\bm{z}-\bm{z'})\geq M \|\bm{z}-\bm{z'}\|_2^2.\label{eq:strongmono}
\end{eqnarray}
Moreover, based on Assumption~\ref{ass:costfun} the operator $T(\bm{z})$ is also Lipschitz continuous, i.e., there exists some constant $L>0$ such that for any $\bm{z},\bm{z'}\in\bm{\cY}\times \bm{\cU}$, we have
\begin{eqnarray}
\|T(\bm{z})-T(\bm{z'})\|_2^2\leq L^2 \|\bm{z}-\bm{z'}\|_2^2. \label{eq:lipschitz}
\end{eqnarray}

\begin{lemma}\label{lem:sigmarelation}
	The following relation holds: $M\leq L$.
\end{lemma}
\begin{proof}
	Apply \eqref{eq:lipschitz} to have
	\begin{eqnarray}
	(T(\bm{z})-T(\bm{z'}))^{\top}(\bm{z}-\bm{z'})&\leq& \|T(\bm{z})-T(\bm{z'})\|_2\|\bm{z}-\bm{z'}\|_2\nonumber\\
	&\leq &L\|\bm{z}-\bm{z'}\|_2^2.\label{eq:intermedia}
	\end{eqnarray}
	Combine \eqref{eq:strongmono} and \eqref{eq:intermedia} to obtain the result.
\end{proof}

Based on the results established so far, we present the next theorem that guarantees the convergence of the primal-dual gradient algorithm~\eqref{eq:mapk} with small enough stepsize. 

\begin{theorem}\label{the:converge}
If the stepsize $\epsilon$ satisfies 
\begin{eqnarray}
0<\epsilon\leq \overline{\epsilon}<2M/L^2\label{eq:converge0}
\end{eqnarray}
for some $\overline{\epsilon}$, \eqref{eq:mapk} converges to the unique saddle point of \eqref{eq:langr} exponentially fast.
\end{theorem}
{We refer to Appendix for the proof.}

\subsection{Motivation for Hierarchical Design}
Note that in \eqref{eq:primaldual} the update of any $p_i$ (resp. $q_i$) involves the knowledge of $\sum_{j\in\cN}R_{ij}(\overline{\mu}_j-\underline{\mu}_j)$ (resp. $\sum_{j\in\cN}X_{ij}(\overline{\mu}_j-\underline{\mu}_j)$). Therefore, at each iteration a coordinator cognizant of the entire network's sensitivity matrices $R$ and $X$ is required to collect updated dual variables from all nodes, calculate $R^{\top}(\overline{\bm{\mu}}-\underline{\bm{\mu}})$ and $X^{\top}(\overline{\bm{\mu}}-\underline{\bm{\mu}})$, and send back the corresponding result to every node. 
This becomes computationally more challenging in a larger network containing thousands of or even more controllable endpoints, not to mention the re-calculation of the large $R,X$ matrices in case changes in network topology or regulator taps occur.

This motivates us to design a hierarchical control structure where the large network is partitioned into smaller subtrees, each managed locally by its regional coordinator (RC), and there is a central coordinator (CC) that only manages a reduced network where each subtree is treated as one node. As we will show in next section, the hierarchical control not only distributes but also reduces a large amount of computation overall.

\section{Hierarchical Distributed Algorithm}\label{sec:hier}

In this section, we explore the tree structure of the distribution network as well as the construction of its sensitivity matrices to facilitate the design of an efficient hierarchical distributed algorithm. To this end, we formally define subtree as follows.

\begin{definition}
A subtree of a tree $\cT$ is a tree consisting of a node in $\cT$, all its descendants in $\cT$, and their connecting lines.
\end{definition}

We categorize all nodes of distribution network $\cT$ into two groups: 
1) $K$ subtrees indexed by $\cT_k=\{\cN_k,\cE_k\},\ k\in\cK=\{1,\ldots,K\}$, and 2) a set $\cN_0$ collecting all the other ``unclustered" nodes in $\cN$.
Here, $\cN_k$ of size $N_k$ is the set of nodes in subtree $\cT_k$ and $\cE_k$ contains their connecting lines. Thus we have $\cup_{k\in\cK}\cN_k\cup\cN_0=\cN\cup\{0\}$ and $\cN_j \cap \cN_k=\emptyset,\forall j\neq k$.
Assume each subtree $\cT_k$ is managed by an RC cognizant of the topology of $\cT_k$ and communicating with all the controllable nodes within $\cT_k$.

Denote the root node of subtree $\cT_k$ by $n_k^0$, and consider a reduced network $\cT^r=\{\cN^r\cup\{0\},\cE^r\}$ where $\cN^r:=\cup_{k\in\cK}\{n_k^0\}\cup \cN_0$ consists of the root nodes of all subtrees and all the unclustered nodes, and $\cE^r$ is the set of their connecting lines. 
We assume there to be a CC cognizant of the topology of the reduced network $\cT^r$ and communicating with all the RCs as well as the unclustered nodes. 

%{\color{blue}
%\begin{remark}
%	Justify subtree in distribution feeder?......
%\end{remark}
%
%}

%Since each indexed subtree is considered as an AG, in this paper, we use the terms ``subtree" and ``AG" interchangeably.

%  \begin{figure}[h]
% 	\centering
% 	\includegraphics[trim = 0mm 0mm 0mm 0mm, clip, scale=0.35]{RijXij3}
% 	\caption{The unclustered nodes and the root nodes of subtrees together with their connecting lines constitute the reduced network. Two subtrees $\cT_h$ and $\cT_k$ share the same $R_{ij}$ and $X_{ij}$ for any of their respective nodes $i$ and $j$.}
% 	\label{fig:RijXij}
% \end{figure}

\subsection{Hierarchical Distributed Algorithm}

For simplicity, we elaborate the algorithm design for real power injections $\bm{p}$ only, and that for $\bm{q}$ follows similarly.

We equivalently rewrite \eqref{eq:primalP} as:
\begin{eqnarray}
    &&\hspace{-8mm}p_i(t+1)=\Big[p_i(t)-\epsilon\Big(\partial_{p_i} C_i(p_i(t),q_i(t))-C'_0(P_0(t))\nonumber\\
    &&\hspace{10mm}+\sum_{j\in\cN}R_{ij}\big(\overline{\mu}_j(t)-\underline{\mu}_j(t)\big)\Big) \Big]_{\cY_i},\ i\in\cN.\label{eq:recastP}
\end{eqnarray}
Note that in \eqref{eq:recastP} while $\partial_{p_i} C_i(p_i(t),q_i(t))$ is local information and the scalar derivative $C'_0(P_0(t))$ can be easily broadcast, the last term $\sum_{j\in\cN}R_{ij}(\overline{\mu}_j(t)-\underline{\mu}_j(t))$ couples the entire network. How to efficiently compute the coupling term is the key to a scalable algorithm. For that purpose, we introduce the following lemma.

\begin{lemma}\label{lem:commonpath}
Given any two subtrees $\cT_h$ and $\cT_k$ with their root nodes $n^0_h$ and $n^0_k$, we have $R_{ij}=R_{n^0_h n^0_k},X_{ij}=X_{n^0_h n^0_k}$, for any $i\in\cN_h$, and any $j\in\cN_k$. Similarly, given any unclustered node $i\in\cN_0$ and a subtree $\cT_k$ with its root node $n^0_k$, we have $R_{ij}=R_{i n^0_k},X_{ij}=X_{i n^0_k}$, for any $j\in\cN_k$.
\end{lemma}
\begin{proof}
We have the two following facts, which are also illustrated in Fig.~\ref{fig:RX}.
First, by \eqref{X_def}, $R_{ij}$ (resp. $X_{ij}$) is the summed resistance (resp. reactance) on the common path of node $i$ and $j$ leading back to node 0. Second, any node in one subtree and any node in another subtree (or any node in one subtree and one unclustered node) share the same common path back to node 0. The result follows immediately.
\end{proof}

Lemma~\ref{lem:commonpath} indicates that the sensitivity matrices $R$ and $X$ possess block structure such that all nodes within one subtree $\cN_h$ share the same sensitivity with respect to all nodes within another subtree $\cN_k$; see Fig.~\ref{fig:complexity} for illustration. This permits a hierarchical distributed way to recalculate the coupling terms.

\begin{figure}
	\centering
	\includegraphics[trim=16mm 0mm 0mm 0mm,clip,scale=0.34]{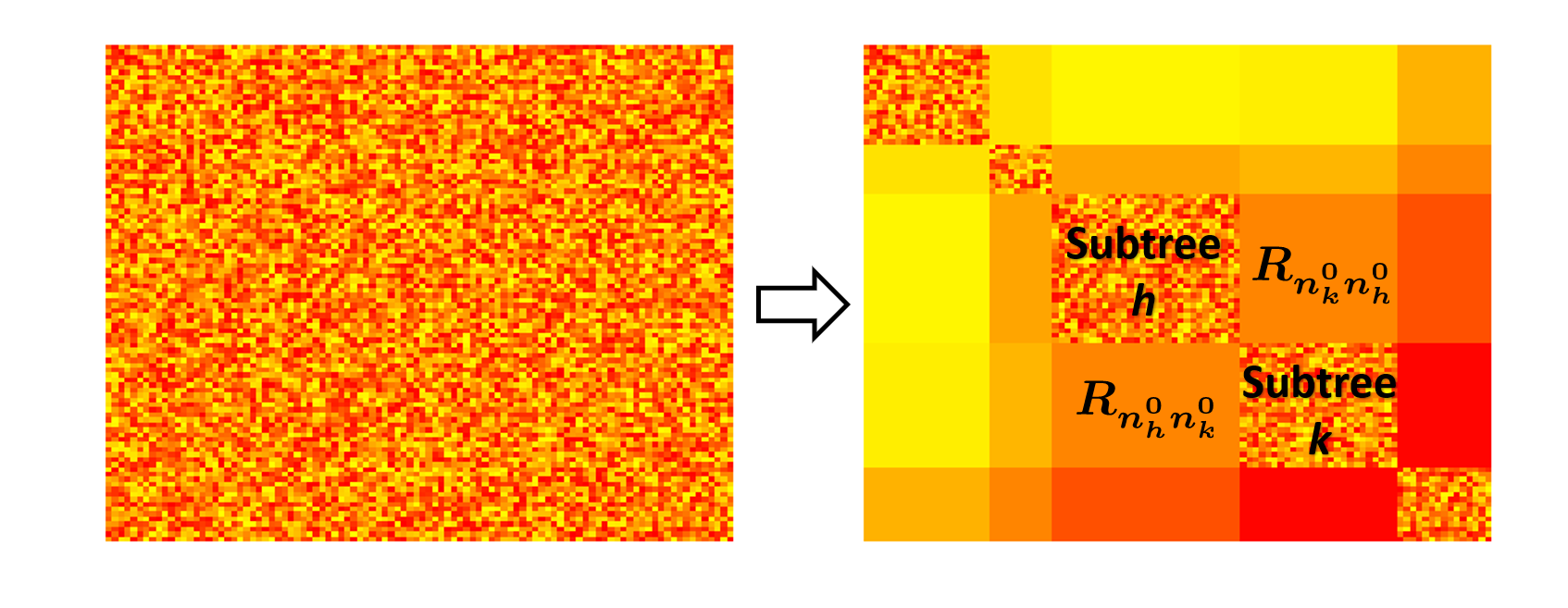}
	\caption{The sensitivity matrix has block structure such that all nodes within one subtree $\cN_h$ share the same sensitivity with respect to all nodes within another subtree $\cN_k$. For simplicity, the set of unclustered nodes $\cN_0$ is assumed to be empty in this illustration.}\label{fig:complexity}
\end{figure}

\textbf{For clustered node $i\in\cN_k$}, we decompose $\sum_{j\in\cN}R_{ij}(\overline{\mu}_j-\underline{\mu}_j)$ as:
\begin{eqnarray}
&& \sum_{j\in\cN_k}R_{ij}(\overline{\mu}_j-\underline{\mu}_j)+\hspace{-2mm}\sum_{j\in\cN\backslash\cN_k}\hspace{-2mm}R_{ij}(\overline{\mu}_j-\underline{\mu}_j)\nonumber\\
&=& \sum_{j\in\cN_k}R_{ij}(\overline{\mu}_j-\underline{\mu}_j)+\hspace{-2mm}\sum_{h\in\cK,h\neq k}\hspace{-3mm}R_{n^0_h n^0_k}\sum_{j\in\cN_h}(\overline{\mu}_j-\underline{\mu}_j)\nonumber\\[-3pt]
&&\hspace{3mm}+\sum_{j\in\cN_0}R_{n_k^0j}(\overline{\mu}_j-\underline{\mu}_j),\label{eq:decompose}
\end{eqnarray}
where the first term of \eqref{eq:decompose} consists of information within subtree $k$ (together with the line parameter from $n_k^0$ to bus 0, i.e., $R_{n_k^0n_k^0}$ and $X_{n_k^0n_k^0}$, which can be informed by CC), the second from all the other subtrees, and the third from the unclustered nodes. 
For convenience, denote by $\alpha_{i}=\sum_{j\in\cN}R_{ij}(\overline{\mu}_j-\underline{\mu}_j)=\alpha_{k,i}^{\text{in}}+\alpha_k^{\text{out}}$ with $\alpha_{k,i}^{\text{in}}$ denoting the first term in \eqref{eq:decompose} and $\alpha_{k,i}^{\text{out}}$ the summation of the second and third terms in \eqref{eq:decompose}.
Note that for $i\in\cN_k$, $\alpha_{k,i}^{\text{in}}$ is accessible by RC $k$ cognizant of the topology of subtree $k$, and that $\alpha_k^{\text{out}}$ does not involve the network structure of any other subtrees but only that of the reduced network $\cT^r$ known by CC. 

%\noindent 
\textbf{For unclustered node $i\in\cN_0$}, similarly we can decompose $\sum_{j\in\cN}R_{ij}(\overline{\mu}_j-\underline{\mu}_j)$ as
\begin{eqnarray}
\sum_{k\in\cK}R_{i n^0_k}\sum_{j\in\cN_k}(\overline{\mu}_j-\underline{\mu}_j)+\sum_{j\in\cN_0}\hspace{-1mm}R_{ij}(\overline{\mu}_j-\underline{\mu}_j),\label{eq:decompose2}
\end{eqnarray}
whose computation only requires the topology and line parameters of the reduced network $\cT^r$ coordinated by CC. 

Eqs.~\eqref{eq:decompose}--\eqref{eq:decompose2} motivate us to design a hierarchical distributed implementation of the primal-dual gradient algorithm~\eqref{eq:primaldual}, where CC and RCs are in charge of only a portion of the whole system: CC manages the reduced network $\cT^r$ and coordinates RCs as well as unclustered nodes without knowing any structural or node-wise information within subtrees, and each RC $k$ manages its own subtrees $\cT_k$ without knowing structural or node-wise information of the other subtrees or the reduced network. This also leads to a secure and privacy-preserving design such that regional topology and node-wise information are both protected. 

We refer the detailed presentation of the single-phase hierarchical algorithm to \cite{zhou2018hierarchical} as its structure is similar to the multi-phase algorithm to be introduced later.

\subsection{Complexity Reduction}\label{sec:complexity}
The hierarchical implementation not only enables parallel computation of the coupling terms in the gradient algorithm, but also largely reduces computational loads  and communication overhead due to the following two reasons:
\begin{itemize}
	\item The term $\sum_{h\in\cK,h\neq k}R_{n^0_h n^0_k}\sum_{j\in\cN_h}(\overline{\mu}_j-\underline{\mu}_j)$ requires less computation than the original $\sum_{j\in\cN_h,h\neq k}R_{ij}(\overline{\mu}_j-\underline{\mu}_j)$;
	\item The intermediate computation result $\alpha_k^{\text{out}}$ is the same for all the nodes in $\cN_k$, reducing a lot of repetitive computations.
\end{itemize}

\subsubsection{Computational Complexity Characterization}
We compare the computational complexity in calculating the coupling term $R^{\top}(\overline{\bm{\mu}}-\underline{\bm{\mu}})$ by centrally coordinated algorithm and hierarchical distributed algorithm as follows.
\begin{itemize}
	\item Centrally coordinated algorithm takes $N^2$ multiplications and $N(N-1)$ additions leading to total complexity of $\mathcal{O}(N^2)$.
	\item For hierarchical distributed algorithm, computational complexity for general network and clustering strategy is challenging to characterize. We instead make simplification and provide intuitions on the computational reduction under the following ideal topological and clustering scenario: Assume that all $N$ nodes with voltage constraints are equally clustered into $K$ subtrees, each with $N_k=N/K$ nodes. This also ignores the third term in \eqref{eq:decompose}.
	\begin{enumerate}
	\item[(a)] To calculate \eqref{eq:decompose} for all $i\in\cN_k$, each RC performs $N_k^2$ multiplications and $N_k(N_k-1)$ additions for the first term, $N_k-1$ additions for the second term, and $N_k$ additions to add the two terms. This results in $2N_k^2+N_k-1$ operations for each RC and $K(2N_k^2+N_k-1)$ for all RCs.
	\item[(b)] CC performs $K^2$ multiplications and $(K-1)K$ additions to put the second term of \eqref{eq:decompose} together.
	\item[(c)] Adding up the results of last two steps and using $N=KN_k$ lead to a total of $2N^2/K+N+2K^2$ operations (ignoring lower-order terms in $K$).
	\item[(d)] Apparently the computation depends on the choice of $K$, e.g., when $K=1$ or $K=N$, the clustering is trivial and  there is no computational reduction. When we choose $K=(N^2/2)^{1/3}$, the total computation burden is minimized to $\mathcal{O}(N^{4/3})$, largely decreased from  $\mathcal{O}(N^{2})$. 
	\end{enumerate}
\end{itemize}

Though in practice such ideal topological and clustering structure is uncommon, a simple geographical or administrative clustering strategy will nevertheless reduce a significant amount of computation. As will be shown in Section V, a 4-subtree clustering of the original network gives us a 4-fold acceleration in convergence speed by reducing computational loads. How to optimally cluster a distribution network remains an open question for our future work.

 \subsection{Multi-Level Distributed Control}\label{sec:multi}
 In large distribution systems, one subtree may still contain too many nodes to be handled efficiently. Meanwhile, some smaller areas within subtrees may want to shelter their topology and node-wise information from RC or CC, e.g., for security concerns.
  
This as well as the fractal structure of tree topology motivates us to consider deeper clustering and multi-level control, i.e., to apply similar approaches to cluster nodes within a subtree into smaller ``sub-subtrees", as illustrated in Fig.~\ref{fig:multilevel}. Mathematically, this is done by decomposing $\alpha_{k,i}^{\text{in}}$ in the same as we do with $\alpha_{i}$ in \eqref{eq:decompose}. Even deeper clustering can be done likewise if necessary. We omit further details here.

\begin{figure}
	\centering
	\includegraphics[scale=0.22]{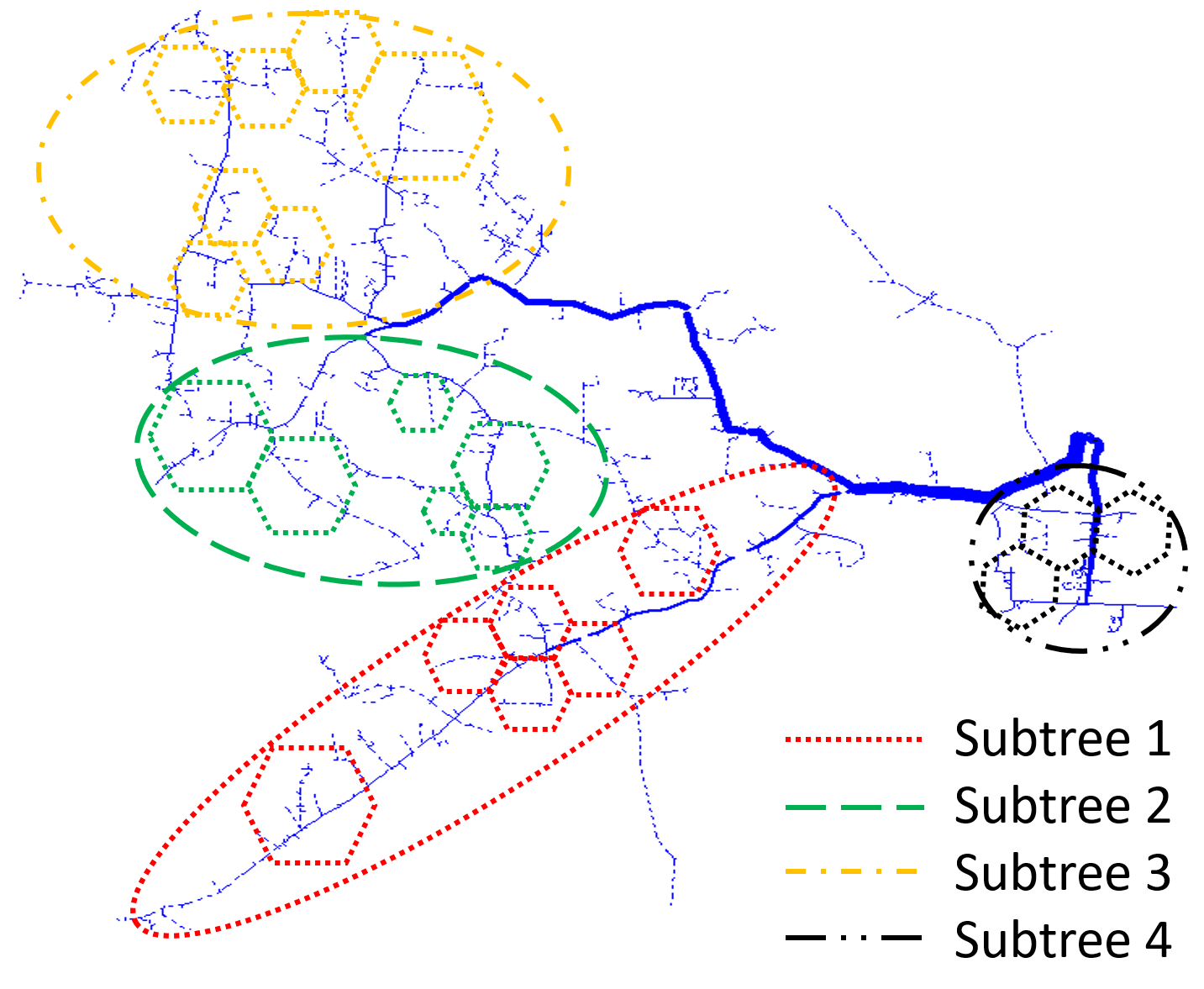}
	\caption{Multi-level control where sub-subtrees are marked by hexagons within each subtree.}\label{fig:multilevel}
\end{figure}

\section{Multi-Phase System Hierarchical Control} \label{sec:multiphase}
Given the intuitions built up from previous sections, we are now ready to design hierarchical distributed algorithms for multi-phase systems.
\subsection{Multi-Phase System Modeling and OPF Formulation}
Define $\mathfrak{i}:=\sqrt{-1}$. Let $a,b,c$ denote the three phases, and $\Phi_i$ the set of phase(s) of node $i\in\cN$, e.g., $\Phi_i=\{a,b,c\}$ for a three-phase node $i$, and $\Phi_j=\{b\}$ for a single b-phase node $j$. Also, in a three-phase system one usually has $\Phi_0=\{a,b,c\}$. Define $\cN^{\phi}\subseteq\cN$ the subset of $\cN$ collecting nodes that have phase $\phi$. Denote by $p_i^{\phi}$, $q_i^{\phi}$, $V_i^{\phi}$ and $v_i^{\phi}$ the real power injection, the reactive power injection, the complex voltage phasor, and the squared voltage magnitude, respectively, of node $i\in\cN$ at phase $\phi\in\Phi_i$. 
Denote by $N_{\Xi}:=\sum_{i\in\cN}|\Phi_i|=\sum_{\phi\in\Phi_0}|\cN^{\phi}|$ the total cardinality of the multi-phase system, where $|\cdot|$ calculates the cardinality of a set. We make the following assumptions to obtain a linearized power flow model for the multi-phase system.

\begin{assumption}\label{ass:3phase}
We consider a multi-phase distribution network where
\begin{enumerate}
    \item[\textbf{1)}] the line losses are small and ignored, and
    \item[\textbf{2)}] the magnitudes of three-phase voltages are approximately equal and the phase differences among three-phase voltages are close to $2\pi/3$, i.e.,
   $\frac{V_i^a}{V_i^b}\approx\frac{V_i^b}{V_i^c}\approx\frac{V_i^c}{V_i^a}\approx e^{\mathfrak{i} 2\pi/3}$.
\end{enumerate}
\end{assumption}
 
% {\color{blue}
%\begin{remark}
%	Note that Assumption~\ref{ass:3phase} is only used to generate a gradient descent direction that can be calculated in a hierarchical distributed manner. Nonlinear model etc... 
%\end{remark}
%}
 
% {\color{cyan} XZ: These assumptions and approximation need well motivated and justified, e.g., in order to obtain sensitivity matrices that can lead to hierarchical design. Also, provide good numerical results.}
 
Let $z_{\zeta\xi}$ be the phase impedance matrix of line $(\zeta,\xi)\in\cE$. For example, if line $(\zeta,\xi)$ has three phases, 
\begin{eqnarray}
z_{\zeta\xi}=\begin{bmatrix}z^{aa}_{\zeta\xi}& z^{ab}_{\zeta\xi}&z^{ac}_{\zeta\xi}\\ z^{ba}_{\zeta\xi}& z^{bb}_{\zeta\xi}&z^{bc}_{\zeta\xi}\\z^{ca}_{\zeta\xi}& z^{cb}_{\zeta\xi}&z^{cc}_{\zeta\xi}\end{bmatrix}\in\mathbb{C}^{3\times 3},\nonumber
\end{eqnarray}
where the off-diagonal elements represent mutual impedance between two phases. Similar to $R_{ij}$ and $X_{ij}$ defined in Eq.~\eqref{X_def}, we define 
\begin{eqnarray}
Z^{\varphi\phi}_{ij}=\hspace{-4mm}\sum_{(\zeta,\xi)\in\cE_i \cap \cE_j}\hspace{-4mm}z^{\varphi\phi}_{\zeta\xi}\hspace{2mm}\in\mathbb{C}\label{eq:defineZ}
\end{eqnarray}
the summarized impedance (if $\varphi=\phi$) or the mutual impedance (if $\varphi\neq\phi$) of the common path of node $i$ and $j$ leading back to node 0, and $\overline{Z}^{\varphi\phi}_{ij}$ its conjugate. 

We denote by $\bm{v}_{_{\Xi}}=[[v_1^{\phi}]^{\top}_{\phi\in\Phi_1}, \ldots, [v_N^{\phi}]^{\top}_{\phi\in\Phi_N}]^{\top}\in\mathbb{R}^{N_{\Xi}}$ the multi-phase squared voltage magnitude vector, and $\bm{p}_{_{\Xi}}=[[p_1^{\phi}]_{\phi\in\Phi_1}^{\top}, \ldots, [p_N^{\phi}]^{\top}_{\phi\in\Phi_N}]^{\top}\in\mathbb{R}^{N_{\Xi}}$ and $\bm{q}_{_{\Xi}}=[[q_1^{\phi}]^{\top}_{\phi\in\Phi_1}, \ldots, [q_N^{\phi}]^{\top}_{\phi\in\Phi_N}]^{\top}\in\mathbb{R}^{N_{\Xi}}$ the multi-phase power injection vectors. We then extend the linearization \eqref{eq:lindistflow} to its multi-phase counterpart written as
\begin{eqnarray}
\bm{v}_{_{\Xi}}=R_{_{\Xi}}\bm{p}_{_{\Xi}}+X_{_{\Xi}}\bm{q}_{_{\Xi}}+\tilde{\bm{v}}_{_{\Xi}},\label{eq:lindistflow_phi}
\end{eqnarray}
where $\tilde{\bm{v}}_{_{\Xi}}\in\mathbb{R}^{N_{\Xi}}$ is a constant vector depending on squared voltage magnitudes at all phases of the slack bus, and the voltage-to-power sensitivity matrices $R_{_{\Xi}}, X_{_{\Xi}}\in\mathbb{R}^{N_{\Xi}\times N_{\Xi}}$ are determined by the linear approximation method developed for multi-phase system \cite{gan2014convex,gan2016online}. The elements of $R_{_{\Xi}}, X_{_{\Xi}}$ are calculated as follows (we use $a=0$, $b=1$, and $c=2$ when calculating $\phi-\varphi$):
\begin{subequations}\label{eq:RX3}
	\begin{eqnarray}
	\partial_{p_j^{\phi}}v_i^{\varphi} \hspace{-2mm}& = & \hspace{0.5mm}  2\mathfrak{Re} \big\{ \overline{Z}^{\varphi\phi}_{ij} \cdot \omega^{\varphi-\phi}\big\},\\
	\partial_{q_j^{\phi}}v_i^{\varphi} \hspace{-2mm} & = &\hspace{-2mm} -2 \mathfrak{Im} \big\{ \overline{Z}^{\varphi\phi}_{ij} \cdot \omega^{\varphi-\phi}\big\},
	\end{eqnarray}
\end{subequations}
for any $\varphi\in\Phi_i$, $\phi\in\Phi_j$, $i,j\in\cN$, with  $\omega=e^{-\mathfrak{i} 2\pi/3}$ and $\mathfrak{Re}\{\cdot\}$ and $\mathfrak{Im}\{\cdot\}$ denoting the real part and the imaginary part of a complex number. Note that when $\varphi=\phi$, Eqs.~\eqref{eq:RX3} coincide with $R_{ij}$ and $X_{ij}$ in Eqs.~\eqref{X_def} for any nodes $i,j\in\cN$; otherwise, Eqs.~\eqref{eq:RX3} calculate the summarized mutual impedance---rotated by phase difference $\pm 2\pi/3$---of the common path of nodes $i,j\in\cN$ leading back to node 0.
 
We further extend Eq.~\eqref{eq:P0} to its multi-phase counterpart as
\begin{eqnarray}
P_0=-\sum_{\phi\in\Phi_0}P_I^{\phi}-\sum_{i\in\cN}\sum_{\phi\in\Phi_i}p_i^{\phi},\label{eq:P0_phi}
\end{eqnarray}
where $P_I^{\phi}$ is the total uncontrollable (inelastic) power injection at phase $\phi\in\Phi_0$. 
We use $\bm{v}_{_{\Xi}}(\bm{p}_{_{\Xi}},\bm{q}_{_{\Xi}})$ and $P_0(\bm{p}_{_{\Xi}})$ to represent Eqs.~\eqref{eq:lindistflow_phi} and \eqref{eq:P0_phi}, and formulate the OPF problem for the multi-phase system as follows:
 \begin{subequations}\label{eq:opt_phi}
 	\begin{eqnarray}
 	\hspace{-5mm}	&\underset{\bm{p}_{_{\Xi}},\bm{q}_{_{\Xi}}}{\min} & \sum_{i\in\cN}\sum_{\phi\in\Phi_i}C_i^{\phi}(p_i^{\phi},q_i^{\phi})+C_0(P_0(\bm{p}_{_{\Xi}})),\\   
 	\hspace{-5mm}	& \text{s.t.}&
 		\underline{v}_i^{\phi} \leq v_i^{\phi}(\bm{p}_{_{\Xi}},\bm{q}_{_{\Xi}}) \leq \overline{v}_i^{\phi}, \phi\in\Phi_i, \forall i\in\cN,\label{eq:voltreg_phi}\\
 	\hspace{-5mm}	&     & (p_i^{\phi},q_i^{\phi})\in\cY_i^{\phi},\phi\in\Phi_i, \forall i\in\cN. \label{eq:X_phi}
 	\end{eqnarray}
 \end{subequations}

Associate dual variables $\overline{\bm{\mu}}_{_{\Xi}}$ and $\underline{\bm{\mu}}_{_{\Xi}}$ with \eqref{eq:voltreg_phi} and we write the regularized Lagrangian of \eqref{eq:opt_phi} as:
\begin{eqnarray}
\cL^{\Phi}_{\eta}(\bm{p}_{_{\Xi}},\bm{q}_{_{\Xi}};\overline{\bm{\mu}}_{_{\Xi}},\underline{\bm{\mu}}_{_{\Xi}})=\sum_{i\in\cN}\sum_{\phi\in\Phi_i}C_i^{\phi}(p_i^{\phi},q_i^{\phi})+C_0(P_0(\bm{p}_{_{\Xi}}))\nonumber\\[-5pt]
+\underline{\bm{\mu}}_{_{\Xi}}^{\top}(\underline{\bm{v}}_{_{\Xi}}-\bm{v}_{_{\Xi}}(\bm{p}_{_{\Xi}},\bm{q}_{_{\Xi}}))+\overline{\bm{\mu}}_{_{\Xi}}^{\top}(\bm{v}_{_{\Xi}}(\bm{p}_{_{\Xi}},\bm{q}_{_{\Xi}})-\overline{\bm{v}}_{_{\Xi}})-\frac{\eta}{2}\|\bm{\mu}_{_{\Xi}}\|^2_2.\label{eq:langr2}\nonumber
\end{eqnarray}

Note that the multi-phase sensitivity matrices $R_{_{\Xi}}$ and $X_{_{\Xi}}$ in Eqs.~\eqref{eq:RX3} have the similar structure as their single-phase counterparts $R$ and $X$ defined in Eqs.~\eqref{X_def}, i.e., the values of $	\partial_{p_j^{\varphi}}v_i^{\phi}$ and $\partial_{q_j^{\varphi}}v_i^{\phi}$ for any $i,j\in\cN$ only depend on the common path of $i$ and $j$ leading back to node 0, as well as their angle difference $\phi-\varphi$. This motivates us to design a similar hierarchical control structure for the multi-phase system.
 
 \subsection{Multi-Phase Hierarchical Distributed Algorithm}
The primal-dual gradient algorithm for solving  the regularized Lagrangian of the convex optimization problem \eqref{eq:opt_phi} reads:
\begin{subequations}\label{eq:primaldual3}
\begin{eqnarray}
\hspace{-6mm}p_i^{\phi}(t+1)\hspace{-2mm}&=&\hspace{-2mm}\Big[p^{\phi}_i(t)-\epsilon\Big(\partial_{p^{\phi}_i} C^{\phi}_i(p^{\phi}_i(t),q^{\phi}_i(t))-\nonumber\\[-2pt]
\hspace{-6mm}&&\hspace{-15mm}C'_0(P_0(t))\hspace{-1mm}+\hspace{-1mm}\sum_{j\in\cN}\sum_{\varphi\in\Phi_j}	\partial_{p_i^{\phi}}v_j^{\varphi}\big(\overline{\mu}^{\varphi}_j(t)-\underline{\mu}^{\varphi}_j(t)\big)\Big) \Big]_{\cY^{\phi}_i},\label{eq:mpPi}\\[-4pt]
\hspace{-6mm}q_i^{\phi}(t+1)\hspace{-2mm}&=&\hspace{-2mm}\Big[q^{\phi}_i(t)-\epsilon\Big(\partial_{q^{\phi}_i} C^{\phi}_i(p^{\phi}_i(t),q^{\phi}_i(t))+\nonumber\\[-2pt]
\hspace{-6mm}&&\hspace{2mm}\sum_{j\in\cN}\sum_{\varphi\in\Phi_j}	\partial_{q_i^{\phi}}v_j^{\varphi}\big(\overline{\mu}^{\varphi}_j(t)-\underline{\mu}^{\varphi}_j(t)\big)\Big) \Big]_{\cY^{\phi}_i},\label{eq:mpQi}\\[-2pt]
\hspace{-6mm}\underline{\mu}^{\phi}_i(t+1)\hspace{-2mm}&=&\hspace{-2mm}[\underline{\mu}^{\phi}_i(t)+\epsilon (\underline{v}^{\phi}_i-v^{\phi}_i(t)-\eta \underline{\mu}^{\phi}_i(t))]_+,\label{eq:mpmu1}\\
\hspace{-6mm}\overline{\mu}^{\phi}_i(t+1)\hspace{-2mm}&=&\hspace{-2mm}[\overline{\mu}^{\phi}_i(t)+\epsilon (v^{\phi}_i(t)-\overline{v}^{\phi}_i-\eta \overline{\mu}^{\phi}_i(t))]_+,\label{eq:mpmu2}\\
\hspace{-6mm}\bm{v}_{_{\Xi}}(t+1)\hspace{-2mm}&=&\hspace{-2mm} R_{_{\Xi}}\bm{p}_{_{\Xi}}(t+1)+X_{_{\Xi}}\bm{q}_{_{\Xi}}(t+1)+\tilde{\bm{v}}_{_{\Xi}},\label{eq:linearphi}\\
\hspace{-6mm}P_0(t+1)\hspace{-2mm}&=&\hspace{-2mm}-\sum_{\phi\in\Phi_0}P_I^{\phi}-\sum_{i\in\cN}\sum_{\phi\in\Phi_i}p_i^{\phi}(t+1),\label{eq:P0phi}
\end{eqnarray}
\end{subequations}
where \eqref{eq:mpPi}--\eqref{eq:mpmu2} are for all $\phi\in\Phi_i$ and all $i\in\cN$. One can obtain similar results as in Theorem~\ref{the:unique}--\ref{the:converge} for the Lagrangian $\cL^{\Phi}_{\eta}$ and the primal-dual gradient algorithm \eqref{eq:primaldual3}. We omit the details here to avoid repetition.

Note that the last terms in the primal update steps \eqref{eq:mpPi}--\eqref{eq:mpQi} not only couple all nodes but also multiple phases together. Nevertheless, similar to Eq.~\eqref{eq:decompose}, we can decompose the coupling terms based on Eqs.~\eqref{eq:defineZ}--\eqref{eq:RX3} along with the radial topology of the network. To this end, we consider the same subtree structure and notations defined in Section~\ref{sec:hier}. We use real power updates for illustration. 

\textbf{For clustered node $i\in\cN_k$}, we have the following decomposition for any $\phi\in\Phi_i$:
\begin{eqnarray}
\hspace{-9mm}&&\hspace{-1mm}2 \sum_{j\in\cN}\sum_{\varphi\in\Phi_j} \mathfrak{Re} \big\{ \overline{Z}^{\varphi\phi}_{ji}\omega^{\varphi-\phi}\big\}\cdot\big(\overline{\mu}^{\varphi}_j(t)-\underline{\mu}^{\varphi}_j(t)\big)\nonumber\\[-3pt]
\hspace{-9mm}&=&\hspace{-1mm}2\mathfrak{Re} \Big\{ \sum_{\varphi\in\Phi_0}\omega^{\varphi-\phi}\sum_{j\in\cN^{\varphi}}  \overline{Z}^{\varphi\phi}_{ji} \big(\overline{\mu}^{\varphi}_j(t)-\underline{\mu}^{\varphi}_j(t)\big)\Big\}
\end{eqnarray}
\begin{eqnarray}
\hspace{-9mm}&=&\hspace{-1mm}2\mathfrak{Re} \Big\{ \sum_{\varphi\in\Phi_0}\omega^{\varphi-\phi}\Big(\hspace{-4mm}\sum_{j\in\cN^{\varphi}\cap\cN_k}\hspace{-4mm}  \overline{Z}^{\varphi\phi}_{ji} \big(\overline{\mu}^{\varphi}_j(t)-\underline{\mu}^{\varphi}_j(t)\big)\nonumber\\
\hspace{-9mm}&&\hspace{-4mm}+\hspace{-8mm}\sum_{\hspace{3mm}j\in\cN^{\varphi}\cap(\hspace{-3mm}\underset{h\in\cK,h\neq k}{\cup}\hspace{-3mm}\cN_h)} \hspace{-7mm} \overline{Z}^{\varphi\phi}_{ji} \big(\overline{\mu}^{\varphi}_j(t)-\underline{\mu}^{\varphi}_j(t)\big)\hspace{-1mm}+\hspace{-5mm}\sum_{j\in\cN^{\varphi}\cap\cN_0} \hspace{-5mm} \overline{Z}^{\varphi\phi}_{ji} \big(\overline{\mu}^{\varphi}_j(t)-\underline{\mu}^{\varphi}_j(t)\big)\hspace{-0.8mm}\Big)\Big\}\hspace{-10mm}\nonumber\\
\hspace{-9mm}&=&\hspace{-1mm}2\mathfrak{Re} \Big\{ \sum_{\varphi\in\Phi_0}\omega^{\varphi-\phi}\Big(\hspace{-4mm}\sum_{j\in\cN^{\varphi}\cap\cN_k}\hspace{-4mm}  \overline{Z}^{\varphi\phi}_{ji} \big(\overline{\mu}^{\varphi}_j(t)-\underline{\mu}^{\varphi}_j(t)\big)\nonumber\\
\hspace{-9mm}&&\hspace{19mm}+\sum_{\substack{n_h^0 \in N^\varphi\\h\in\cK,h\neq k }}\hspace{-3mm}\overline{Z}^{\varphi\phi}_{n_h^0 n_k^0}\hspace{-4mm}\sum_{j\in\cN^{\varphi}\cap\cN_h}\hspace{-4mm}\big(\overline{\mu}^{\varphi}_j(t)-\underline{\mu}^{\varphi}_j(t)\big)\nonumber\\
\hspace{-9mm}&&\hspace{19mm}+\sum_{j\in\cN^{\varphi}\cap\cN_0}\hspace{-4mm}\overline{Z}^{\varphi\phi}_{j n_k^0}\big(\overline{\mu}^{\varphi}_j(t)-\underline{\mu}^{\varphi}_j(t)\big)\Big)\Big\},\label{eq:decompose3}
\end{eqnarray}
with $Z^{\varphi\phi}_{n_h^0 n_k^0}=\sum_{(\zeta,\xi)\in\cE_{n_k^0} \cap \cE_{n_h^0}} z^{\varphi\phi}_{\zeta\xi}$ summarizing the (mutual) impedance on the common path of subtree $h$ and subtree $k$ and $Z^{\varphi\phi}_{j n_k^0}=\sum_{(\zeta,\xi)\in\cE_j \cap \cE_{n_k^0}} z^{\varphi\phi}_{\zeta\xi}$ summarizing the (mutual) impedance on the common path of unclustered node $j$ and subtree $k$.

\textbf{For unclustered node $i\in\cN_0, \forall\phi\in\Phi_i$}, similarly we have:
\begin{eqnarray}
\hspace{-8mm}&&2 \sum_{j\in\cN}\hspace{-1mm}\sum_{\varphi\in\Phi_j} \mathfrak{Re} \big\{ \overline{Z}^{\varphi\phi}_{ji}  \omega^{\varphi-\phi}\big\}\cdot\big(\overline{\mu}^{\varphi}_j(t)-\underline{\mu}^{\varphi}_j(t)\big)\nonumber\\
\hspace{-8mm}&=&2\mathfrak{Re} \Big\{ \sum_{\varphi\in\Phi_0}\omega^{\varphi-\phi}\Big(\sum_{\substack{n_k^0 \in N^\varphi\\k\in\cK }}\overline{Z}^{\varphi\phi}_{n_k^0 i}\hspace{-4mm}\sum_{j\in\cN^{\varphi}\cap\cN_k}\hspace{-4mm}\big(\overline{\mu}^{\varphi}_j(t)-\underline{\mu}^{\varphi}_j(t)\big)\nonumber\\
\hspace{-8mm}&&\hspace{22mm}+\sum_{j\in\cN^{\varphi}\cap\cN_0} \hspace{-4mm} \overline{Z}^{\varphi\phi}_{ji} \big(\overline{\mu}^{\varphi}_j(t)-\underline{\mu}^{\varphi}_j(t)\big)\Big)\Big\}.\label{eq:decompose4}
\end{eqnarray}

One can apply similar approaches to obtain the decomposed results to update reactive power injections. Then the resultant equivalent form of Eqs.~\eqref{eq:primaldual3} can be implemented in a hierarchical distributed way with the collaboration of RCs and CC. We summarize the resultant algorithm for multi-phase systems based on the primal-dual gradient algorithm \eqref{eq:primaldual3} and the decomposition strategies \eqref{eq:decompose3}--\eqref{eq:decompose4} in Algorithm~\ref{alg:distalg2} where we use $s^{\phi}_i(t)=\alpha^{\phi}_i(t)+\mathfrak{i} \beta^{\phi}_i(t)$ for notational simplicity. We also use $s_k^{\phi_\text{out}}$ and $s^{\phi_\text{in}}_{k,i}$ to denote the parts of $s^{\phi}_i(t)$ that are calculated outside of subtree $k$ and within subtree $k$, respectively. 
See Fig.~\ref{fig:testfeeder} for an illustration of clustering and information flows in Algorithm~\ref{alg:distalg2}. 
Due to their mathematical equivalence,   Algorithm~\ref{alg:distalg2} and Eqs.~\eqref{eq:primaldual3} share the same dynamic properties, which will also be illustrated in Section~\ref{sec:numerical}.

\begin{algorithm*}[ht]
	\caption{Hierarchical Distributed Voltage Regulation for Multi-Phase Distribution Netoworks} 
	\begin{algorithmic}\label{alg:distalg2}
		
	\REPEAT
		
	\STATE[1] node $i\in\cN$ sets $\alpha^{\phi}_i(t)=\mathfrak{Re}\{s^{\phi}_i(t)\}, \beta^{\phi}_i(t)=-\mathfrak{Im}\{s^{\phi}_i(t)\}$, and updates $(p^{\phi}_i(t+1),q^{\phi}_i(t+1))$ for all $\phi\in\Phi_i$ by
		\begin{subequations}
			\begin{eqnarray}
			p^{\phi}_i(t+1)&=&[p^{\phi}_i(t)-\epsilon(\partial_{p^{\phi}_i} C^{\phi}_i(p^{\phi}_i(t),q^{\phi}_i(t)) - C'_0(P_0(t)) +\alpha^{\phi}_{i}(t))]_{\cY^{\phi}_i},\\ q^{\phi}_i(t+1)&=&[q^{\phi}_i(t)-\epsilon(\partial_{q^{\phi}_i} C^{\phi}_i(p^{\phi}_i(t),q^{\phi}_i(t)) +\beta^{\phi}_{i}(t))]_{\cY^{\phi}_i},
			\end{eqnarray}
		\end{subequations}
		
		and $(\underline{\mu}^{\phi}_i(t+1),\overline{\mu}^{\phi}_i(t+1))$ for all $\phi\in\Phi_i$ based on local voltage magnitude by
		\begin{subequations}
			\begin{eqnarray}
			\underline{\mu}^{\phi}_i(t+1)&=&[\underline{\mu}^{\phi}_i(t)+\epsilon (\underline{v}^{\phi}_i-v^{\phi}_i(t)-\eta \underline{\mu}^{\phi}_i(t))]_+,\\ \overline{\mu}^{\phi}_i(t+1)&=&[\overline{\mu}^{\phi}_i(t)+\epsilon (v^{\phi}_i(t)-\overline{v}^{\phi}_i-\eta \overline{\mu}^{\phi}_i(t))]_+.
			\end{eqnarray}
		\end{subequations}

	\STATE[2] RC $k\in\cK$ sends $\sum_{j\in\cN^{\varphi}\cap\cN_k}\big(\overline{\mu}^{\varphi}_j(t)-\underline{\mu}^{\varphi}_j(t)\big), \forall \varphi\in\Phi_{n_k^0}$ to CC, and 
		unclustered node $j\in\cN_0$ sends $\big(\overline{\mu}^{\varphi}_j(t)-\underline{\mu}^{\varphi}_j(t)\big), \forall \varphi\in\Phi_j$ to CC.

	\STATE[3] CC computes within the reduced network $\forall \phi\in\Phi_0$:
	\begin{eqnarray}
		\hspace{-4mm}s_k^{\phi_\text{out}}(t+1)&=&2 \sum_{\varphi\in\Phi_0}\omega^{\varphi-\phi}\Big(\sum_{\substack{n_h^0 \in N^\varphi\\h\in\cK,h\neq k }}\overline{Z}^{\varphi\phi}_{n_h^0 n_k^0}\hspace{-4mm}\sum_{j\in\cN^{\varphi}\cap\cN_h}\hspace{-4mm}\big(\overline{\mu}^{\varphi}_j(t)-\underline{\mu}^{\varphi}_j(t)\big)+\hspace{-4mm}\sum_{j\in\cN^{\varphi}\cap\cN_0}\hspace{-4mm}\overline{Z}^{\varphi\phi}_{j n_k^0}\big(\overline{\mu}^{\varphi}_j(t)-\underline{\mu}^{\varphi}_j(t)\big)\Big),\  \forall k\in\cK,\\
			s^{\phi}_i(t+1)&=&2 \sum_{\varphi\in\Phi_0}\omega^{\varphi-\phi}\Big(\sum_{\substack{n_k^0 \in N^\varphi\\k\in\cK}}\overline{Z}^{\varphi\phi}_{n_k^0 i}\hspace{-4mm}\sum_{j\in\cN^{\varphi}\cap\cN_k}\hspace{-4mm}\big(\overline{\mu}^{\varphi}_j(t)-\underline{\mu}^{\varphi}_j(t)\big)+\hspace{-4mm}\sum_{j\in\cN^{\varphi}\cap\cN_0} \hspace{-4mm} \overline{Z}^{\varphi\phi}_{ji} \big(\overline{\mu}^{\varphi}_j(t)-\underline{\mu}^{\varphi}_j(t)\big)\Big),\ \forall i\in\cN_0,
	\end{eqnarray}
	and sends $[s_k^{\phi_\text{out}}(t+1)]^{\top}_{\phi\in\Phi_0}$ to RC $k\in\cK$,  and 
	$[s^{\phi}_i(t+1)]^{\top}_{\phi\in\Phi_0}$
		to node $i\in\cN_0$.
		
		\STATE[4] RC $k\in\cK$ calculates $s^{\phi_\text{in}}_{k,i}$ and $s^{\phi}_{i}$ within subtree $k$, $\forall i\in\cN_k,\ \forall \phi\in\Phi_i$ by
		\begin{eqnarray}
			s^{\phi_\text{in}}_{k,i}(t+1)&=&2\sum_{\varphi\in\Phi_0}\omega^{\varphi-\phi}\hspace{-4mm}\sum_{j\in\cN^{\varphi}\cap\cN_k}\hspace{-4mm}  \overline{Z}^{\varphi\phi}_{ji} \big(\overline{\mu}^{\varphi}_j(t)-\underline{\mu}^{\varphi}_j(t)\big),\\
			s^{\phi}_{i}(t+1)&=&s^{\phi_\text{in}}_{k,i}(t+1)+s^{\phi_\text{out}}_k(t+1),
		\end{eqnarray}
		and sends 
		$[s^{\phi}_i(t+1)]^{\top}_{\phi\in\Phi_i}$ to node $i\in\cN_k$.

	\STATE[5] $\bm{v}_{_{\Xi}}(t+1)$ and $P_0(t+1)$ are updated by:
		\begin{eqnarray}
		\bm{v}_{_{\Xi}}(t+1)&=&R_{_{\Xi}}\bm{p}_{_{\Xi}}(t+1)+X_{_{\Xi}}\bm{q}_{_{\Xi}}(t+1)+\tilde{\bm{v}}_{_{\Xi}},\\ P_0(t+1)&=&-\sum_{\phi\in\Phi_0}P_I^{\phi}-\sum_{i\in\cN}\sum_{\phi\in\Phi_i}p_i^{\phi}(t+1).
		\end{eqnarray}
		
	\STATE[6] CC computes/measures $P_0(t+1)$ at the substation and broadcasts $C'_0(P_0(t+1))$.
		
	\UNTIL stopping criterion is met (e.g., $|P^{\phi}_0(\bm{p}^{\phi}(t+1))-P^{\phi}_0(\bm{p}^{\phi}(t))|<\sigma, \ \forall \phi\in\Phi_0$ for some small $\sigma>0$)
	\end{algorithmic}
\end{algorithm*}

\begin{figure}
	\centering
	\includegraphics[scale=0.25]{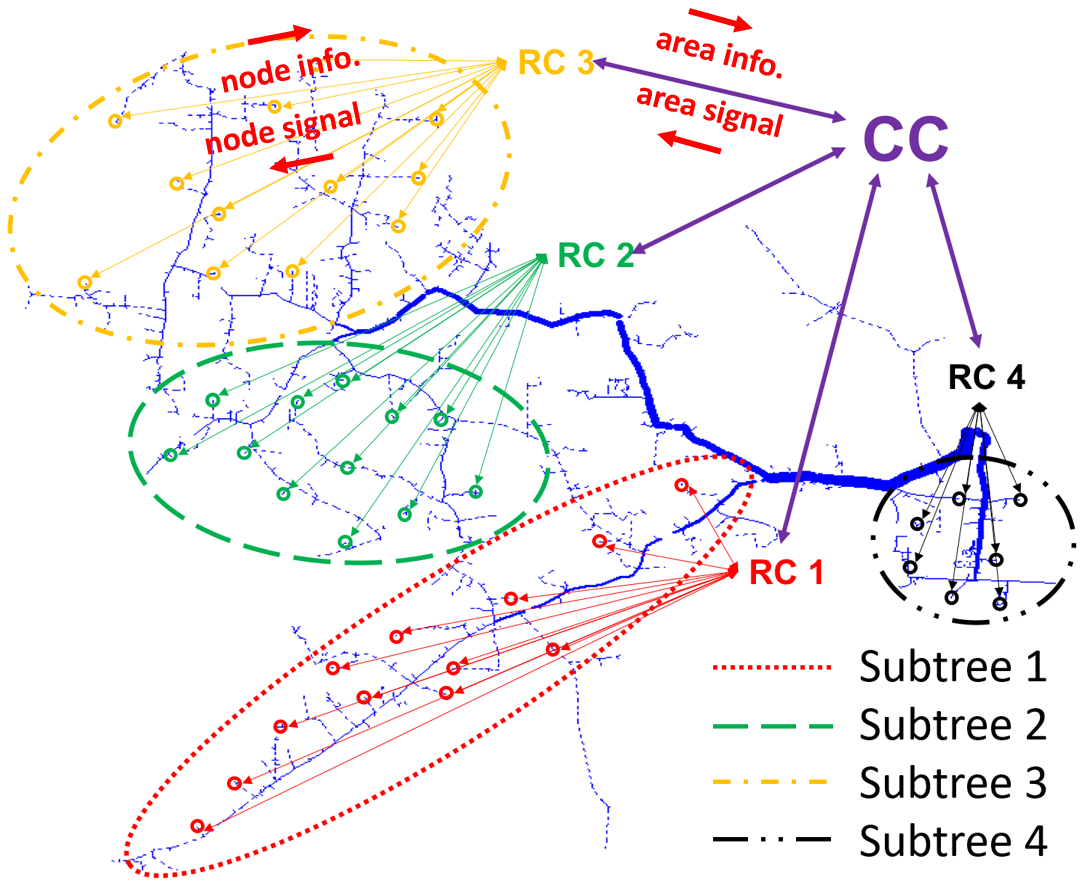}
	\caption{11000-node test feeder constructed from an IEEE 8500-node test feeder and a modified EPRI Test Circuits Ckt7. Four subtrees are formed for our experiments.}\label{fig:testfeeder}
\end{figure}

\subsection{Nonlinear Power Flow Analysis}
In this part, we implement Algorithm~\ref{alg:distalg2} with nonlinear power flow, which can be either balanced or unbalanced, and characterize its performance. To this end, we denote the voltage magnitudes and the total active power loads updated by the nonlinear multi-phase power flow equations by $(\hat{\bm{v}}_{_{\Xi}},\hat{P}_0)=\mathcal{F}(\bm{p}_{_{\Xi}},\bm{q}_{_{\Xi}})$, where the nonlinear power flow $\mathcal{F}$ can be updated based on some specific model from literature or can be measured from physical power flow.

\begin{assumption}\label{ass:model}
	The discrepancy between the linearized model and the nonlinear model is bounded, i.e., there exists some constants $e_1,e_2>0$ such that given any $(\bm{p}_{_{\Xi}},\bm{q}_{_{\Xi}})\in\bm{\cY}$ we have
	\begin{eqnarray}
	\|\bm{v}_{_{\Xi}}(\bm{p}_{_{\Xi}},\bm{q}_{_{\Xi}})-	\hat{\bm{v}}_{_{\Xi}}(\bm{p}_{_{\Xi}},\bm{q}_{_{\Xi}})\|_2&\leq& e_1,\nonumber\\
	|P_0(\bm{p}_{_{\Xi}})-\hat{P}_0(\bm{p}_{_{\Xi}},\bm{q}_{_{\Xi}})|&\leq& e_2.\nonumber
	\end{eqnarray}
\end{assumption}

%The bound $e$ is expected to be small as the linearized model \eqref{eq:lindistflow_phi} is accurate enough under normal operating condition.

Similar to Section~\ref{sec:conv}, we define $\bm{\mu}_{_{\Xi}} :=[\underline{\bm{\mu}}_{_{\Xi}}^{\top}, \overline{\bm{\mu}}_{_{\Xi}}^{\top}]^{\top}$ and $\bm{z}_{_{\Xi}}:=[\bm{p}_{_{\Xi}}^{\top},\bm{q}_{_{\Xi}}^{\top},\bm{\mu}_{_{\Xi}}^{\top}]^{\top}$ to equivalently rewrite the primal-dual gradient dynamics \eqref{eq:primaldual3} as 
\begin{eqnarray}
&&\bm{z}_{_{\Xi}}(t+1)=\left[\bm{z}_{_{\Xi}}(t)-\epsilon T_{_{\Xi}}(\bm{z}_{_{\Xi}}(t)) \right]_{\bm{\cY}\times \bm{\cU}}\hspace{-4mm},\label{eq:mapk2}
\end{eqnarray}
with its gradient operators denoted by 
\begin{eqnarray}
&&T_{_{\Xi}}(\bm{z}_{_{\Xi}}):=\begin{bmatrix}\nabla_{\bm{y}_{_{\Xi}}} \hL^{\Phi}_{\eta}(\bm{y}_{_{\Xi}}(t), \bm{\mu}_{_{\Xi}}(t)) \\ -\nabla_{\bm{\mu}_{_{\Xi}}}\hL^{\Phi}_{\eta}(\bm{y}_{_{\Xi}}(t), \bm{\mu}_{_{\Xi}}(t))
\end{bmatrix}.\nonumber
\end{eqnarray}

Note that Lipschitz continuity and strong monotonicity hold for $T_{_{\Xi}}(\bm{z}_{_{\Xi}})$ with positive constants $L_{_{\Xi}}$ and $M_{_{\Xi}}$ for any feasible $\bm{z}_{_{\Xi}},\bm{z}'_{_{\Xi}}$ as follows:
\begin{subequations}\label{eq:multiproperty}
	\begin{eqnarray}
	\hspace{-8mm}&&(T_{_{\Xi}}(\bm{z}_{_{\Xi}})-T_{_{\Xi}}(\bm{z_{_{\Xi}}'}))^{\top}(\bm{z}_{_{\Xi}}-\bm{z}_{_{\Xi}}')\geq M_{_{\Xi}} \|\bm{z}_{_{\Xi}}-\bm{z_{_{\Xi}}}'\|_2^2,\label{eq:multiproperty1}\\
	\hspace{-8mm}&&\|T_{_{\Xi}}(\bm{z}_{_{\Xi}})-T_{_{\Xi}}(\bm{z}_{_{\Xi}}')\|_2^2\leq L_{_{\Xi}}^2 \|\bm{z}_{_{\Xi}}-\bm{z_{_{\Xi}}}'\|_2^2, \label{eq:multiproperty2}\\
	\hspace{-8mm}&&M_{_{\Xi}} \leq L_{_{\Xi}}.\label{eq:multiproperty3}
	\end{eqnarray}
\end{subequations}

To implement the gradient algorithm with nonlinear power flow, we replace $C'_0(P_0(t))$ in \eqref{eq:mpPi} with $C'_0(\hat{P}_0(t))$, $v^{\phi}_i(t)$ in \eqref{eq:mpmu1}--\eqref{eq:mpmu2} with $\hat{v}^{\phi}_i(t)$, and \eqref{eq:linearphi}--\eqref{eq:P0phi} with $(\hat{\bm{v}}_{_{\Xi}}(t+1),\hat{P}_0(t+1))=\mathcal{F}(\bm{p}_{_{\Xi}}(t+1),\bm{q}_{_{\Xi}}(t+1))$. In other words, the values of $\hat{\bm{v}}_{_{\Xi}},\hat{P}_0$ are updated by the nonlinear power flow while the gradient of $\hat{\bm{v}}_{_{\Xi}},\hat{P}_0$ with respect to decision variables are calculated based on the linearized model \eqref{eq:lindistflow_phi}--\eqref{eq:P0_phi}. Similar model-based feedback control has been applied and characterized with provable performance in recent literature, e.g., \cite{colombino2019online,dall2018optimala,zhou2017incentive}. 

We denote the gradient operator with nonlinear power flow by $\hat{T}_{_{\Xi}}(\bm{z}_{_{\Xi}})$, and the resultant projected gradient algorithm as:
\begin{eqnarray}
&&\bm{z}_{_{\Xi}}(t+1)=\left[\bm{z}_{_{\Xi}}(t)-\epsilon \hat{T}_{_{\Xi}}(\bm{z}_{_{\Xi}}) \right]_{\bm{\cY}\times \bm{\cU}}.\label{eq:mapk3}
\end{eqnarray}

Based on Assumption~\ref{ass:model} and Lipschitz continuity of our gradient operator, we must have
\begin{eqnarray}
\|T_{_{\Xi}}(\bm{z}_{_{\Xi}})-\hat{T}_{_{\Xi}}(\bm{z}_{_{\Xi}})\|^2_2\leq \rho\label{eq:rho}
\end{eqnarray}
for some constant $\rho>0$. We characterize the convergence of dynamics \eqref{eq:mapk3} next.

\begin{theorem}\label{the:nonlinear}
	Under Assumptions~\ref{ass:costfun}--\ref{ass:model} and a stepsize $\epsilon$ chosen according to
	\begin{eqnarray}
		0<\epsilon\leq \bar{\epsilon}<2M_{_{\Xi}}/L_{_{\Xi}}^2\label{eq:stepsize}
	\end{eqnarray}
	for some $\bar{\epsilon}$, dynamics \eqref{eq:mapk3} converges to the saddle point $\bm{z}_{_{\Xi}}^*$ of $\cL^{\Phi}_{\eta}$ as
	\begin{eqnarray}
	\lim_{t\rightarrow\infty}\sup\|\bm{z}_{_{\Xi}}(t)-\bm{z}_{_{\Xi}}^*\|_2^2 = \frac{\rho}{2M_{_{\Xi}}/\epsilon -L_{_{\Xi}}^2}.\label{eq:converge}
	\end{eqnarray}
\end{theorem}
\begin{proof}
The distance between $\bm{z}_{_{\Xi}}(t+1)$ and $\bm{z}_{_{\Xi}}^*$ can be calculated as:
\begin{eqnarray}
\hspace{-3mm}&&\|\bm{z}_{_{\Xi}}(t+1)-\bm{z}_{_{\Xi}}^*\|_2^2\nonumber\\
\hspace{-3mm}&\leq& \|\bm{z}_{_{\Xi}}(t)-\epsilon \hat{T}(\bm{z}_{_{\Xi}}(t))-\bm{z}_{_{\Xi}}^*+\epsilon T(\bm{z}_{_{\Xi}}^*)\|_2^2\nonumber\\
\hspace{-3mm}&=& \|\bm{z}_{_{\Xi}}(t)-\epsilon {T}(\bm{z}_{_{\Xi}}(t))+\epsilon {T}(\bm{z}_{_{\Xi}}(t))-\epsilon \hat{T}(\bm{z}_{_{\Xi}}(t))\nonumber\\
\hspace{-3mm}&&\hspace{3mm}-\bm{z}_{_{\Xi}}^*+\epsilon T(\bm{z}_{_{\Xi}}^*)\|_2^2\nonumber
\end{eqnarray}
\begin{eqnarray}
\hspace{-3mm}&\leq& \|\bm{z}_{_{\Xi}}(t)-\epsilon {T}(\bm{z}_{_{\Xi}}(t))-\bm{z}_{_{\Xi}}^*+\epsilon T(\bm{z}_{_{\Xi}}^*)\|_2^2+\epsilon^2\rho\nonumber\\
\hspace{-3mm}&=& \|\bm{z}(t)-\bm{z}^*\|_2^2+\|\epsilon T(\bm{z}(t))-\epsilon T(\bm{z}^*)\|_2^2\nonumber\\
\hspace{-3mm}&&\hspace{3mm}-2\epsilon (\bm{z}(t)-\bm{z}^*)^{\top}(T(\bm{z}(t))- T(\bm{z}^*))+\epsilon^2\rho.\nonumber\\
\hspace{-3mm}&\leq&\Delta\|\bm{z}(t)-\bm{z}^*\|_2^2+\epsilon^2\rho\nonumber\\
\hspace{-3mm}&\leq& \Delta^t\|\bm{z}(1)-\bm{z}^*\|_2^2+\epsilon^2\rho\frac{1-\Delta^t}{1-\Delta}\nonumber
\end{eqnarray}
with $\Delta=1+\epsilon^2L_{_{\Xi}}^2-2\epsilon M_{_{\Xi}}$. In the above, the first inequality comes from non-expansiveness of projection operator, the second from \eqref{eq:rho}, the third from \eqref{eq:multiproperty1}--\eqref{eq:multiproperty2}, and the last inequality is obtained by recursively executing previous steps. Then based on \eqref{eq:multiproperty3} and condition~\eqref{eq:stepsize} one has 
%$\|\bm{z}(t+1)-\bm{z}^*\|_2^2\leq \Delta \|\bm{z}(t)-\bm{z}^*\|_2^2$ for some constant
 $0<\Delta\leq\bar{\Delta} <1$ for some $\bar{\Delta}$ such that when $t\rightarrow\infty$, \eqref{eq:converge} follows.
%\begin{eqnarray}
%\lim_{t\rightarrow\infty}\sup\|\bm{z}_{_{\Xi}}(t)-\bm{z}_{_{\Xi}}^*\|_2^2 = \frac{\epsilon^2\rho}{1-\Delta}.\nonumber
%\end{eqnarray}
\end{proof}

Theorem~\ref{the:nonlinear} indicates that, despite the discrepancy introduced by linear and nonlinear power flows, dynamics~\eqref{eq:mapk3} converges to the saddle point of $\cL^{\Phi}_{\eta}$ with bounded distance $\frac{\rho}{2M_{_{\Xi}}/\epsilon -L_{_{\Xi}}^2}$. Therefore, given any accuracy requirement $\delta>0$, if stepsize is chosen such that $\epsilon<\min\{\frac{2M_{_{\Xi}}\delta}{\rho+L^2_{_{\Xi}}\delta},\frac{2M_{_{\Xi}}}{L^2_{_{\Xi}}}\}$, we have 
\begin{eqnarray}
\lim_{t\rightarrow\infty}\sup\|\bm{z}_{_{\Xi}}(t)-\bm{z}_{_{\Xi}}^*\|_2^2<\delta.
\end{eqnarray}

\section{Numerical Results}\label{sec:numerical}

A three-phase unbalanced, 11,000-node test feeder is constructed by connecting an IEEE 8,500-node test feeder and a modified EPRI Ckt7 test feeder at the substation. Fig.~\ref{fig:testfeeder} shows the single-line diagram of the feeder with its line width proportional to the nominal power flow on it. The primary side of the feeder is modeled in detail, while the loads on the secondary side are lumped into corresponding distribution transformers, resulting in a 4,521-node network with 1,043 controllable (aggregated) loads. 
We group all the nodes into unclustered nodes and four subtrees marked in Fig.~\ref{fig:testfeeder}. Subtrees 1--4 contain 357, 222, 310, and 154 nodes with controllable loads, respectively. We fix the loads on all 292 unclustered nodes for simplicity.

The three-phase unbalanced nonlinear power flow model is simulated in OpenDSS. With default control of capacitors and regulators in OpenDSS \cite{dugan2010ieee}, one achieves the voltage profile shown in Fig.~\ref{fig:Voltage_controlled} with orange dots, where under-voltages are observed.
We next disable the control of all the capacitors and regulators to obtain the heavily under-voltage scenario marked with blue dots in Fig.~\ref{fig:Voltage_controlled}. We implement Algorithm~\ref{alg:distalg2} with this scenario as the initial condition.

The simulation is conducted on a laptop with Intel Core i7-7600U CPU @ 2.80GHz 2.90GHz, 8.00GB RAM, running Python 3.6 on Windows 10 Enterprise Version.

\subsection{Numerical Performance Evaluation}
For each controllable node $i$ of phase $\phi$, we consider minimizing the cost of its deviation from its nominal (most preferred) load level $(p_i^{\phi}(0), q_i^{\phi}(0))$, i.e., $C^{\phi}_i(p^{\phi}_i,q^{\phi}_i)=(p^{\phi}_i-p_i^{\phi}(0))^2+(q^{\phi}_i-q_i^{\phi}(0))^2$. We focus on voltage regulation here. So, we set $C_0(P_0)$ to $0.0005(P_0-\tilde{P}_0)^2$ with a small weight and $\tilde{P}_0=P_0(0)$. $\underline{v}_i$ and $\overline{v}_i$ are uniformly set to $0.95$~p.u. and $1.05$~p.u., respectively. We implement Algorithm~\ref{alg:distalg2} with a constant stepsize $3.5\times 10^{-4}$ for the primal update, and $3.5\times 10^{-3}$ for the dual.

 \begin{figure}
     \centering
     \includegraphics[scale=0.29]{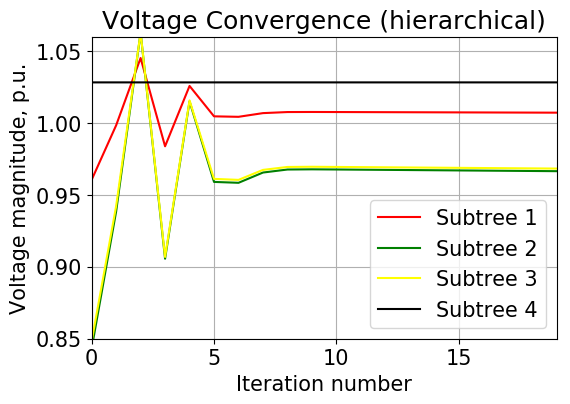}
     \includegraphics[scale=0.29]{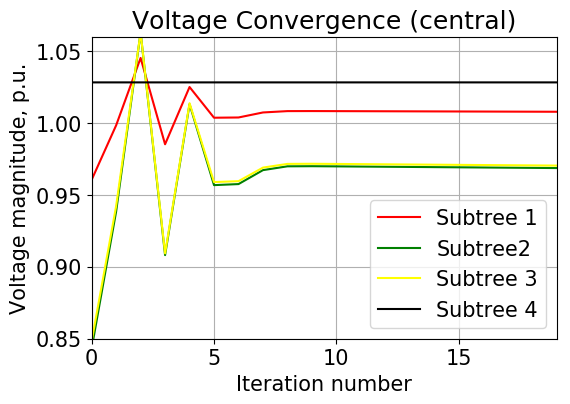}
     \caption{The hierarchical distributed algorithm (left) and the centrally coordinated algorithm (right) show identical convergence dynamics under the same setup.}\label{fig:Voltage_converge}
 \end{figure}
 
\subsubsection{Convergence and Computational Speed Improvement}
It takes about 1730 iterations to reach 1\% of the optimal value and 3000 iterations to reach the optimal; see the red curve in Fig.~\ref{fig:Voltage_controlled} (right). This result is also comparable to those in literature, e.g., \cite{peng2018distributed}.

The convergence dynamics of the centrally coordinated algorithm is identical to the hierarchical distributed implementation as sampled and illustrated in Fig.~\ref{fig:Voltage_converge}. However, due to the overall complexity reduction, the computational time per iteration is reduced from more than 4s to 1s. Meanwhile, if we consider parallel computation\footnote{We imitate  the performance of parallel computation on one PC by timing the computation for CC and each RC. If the algorithm is to be implemented on multiple devices, we also need consider communication delay which is ignored here.}  and takes the slowest cluster to estimate the overall time consumed, each iteration only takes 0.37s. Therefore, the overall computational speed improvement is more than 10 folds without compromising any accuracy of the OPF solution.
Noticing that the computational time at each cluster is approximately proportional to the square of node number, we  expect faster performance if more clusters, each with a smaller number of nodes, are divided. 
 \begin{figure}
	\centering
	\includegraphics[scale=0.5]{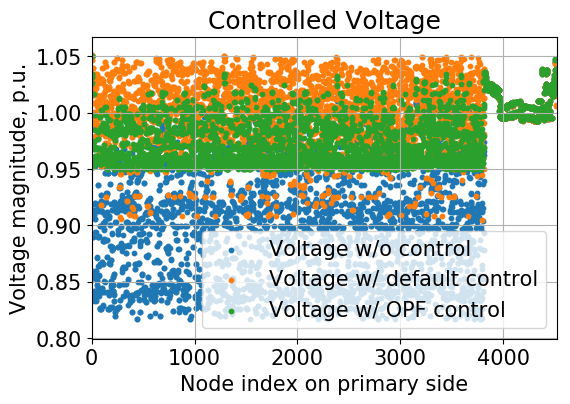}
	\caption{Voltages are strictly controlled within $[0.95,1.05]$ p.u. by the hierarchical distributed algorithm.}\label{fig:Voltage_controlled}
\end{figure}
 \begin{figure}
	\centering
	\includegraphics[scale=0.5]{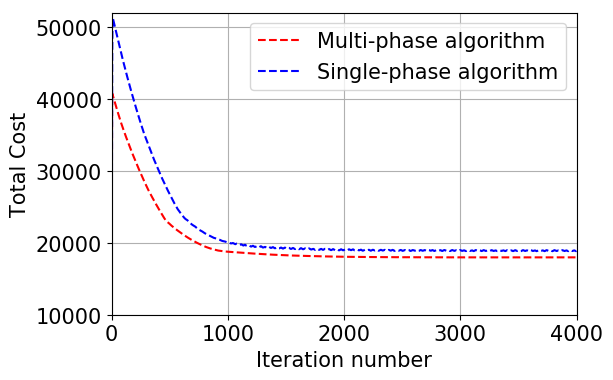}
	\caption{The minimal costs achieved by multi-phase algorithm and single-phase algorithm.}\label{fig:totalcost}
\end{figure}
\subsubsection{Voltage Regulation} We plot the regulated voltages obtained by Algorithm~\ref{alg:distalg2} with green dots in Fig.~\ref{fig:Voltage_controlled}. Note that the voltage magnitudes of all nodes are \textit{strictly} constrained within the $[0.95,1.05]$ p.u. bound. In contrast, the default control of regulators and capacitors cannot guarantee that all the voltages are within this bound.

\subsubsection{Comparison with Single-Phase Algorithm}
We apply the single-phase algorithm with the same setup for comparison. The detailed results of the single-phase algorithm are referred to \cite{zhou2018hierarchical}. The results show that it takes multi-phase algorithm more time to execute each iteration (1s v.s 0.45s), which is expected since the multi-phase algorithm executes more computation by considering inter-phase sensitivity.
On the other hand, thanks to more accurate linearization, the minimal cost obtained by the multi-phase algorithm is 17971, 4.6\% smaller than the value of 18840 obtained by the single-phase algorithm in \cite{zhou2018hierarchical}, as illustrated in Fig.~\ref{fig:totalcost}. This result also echoes with Eq.~\eqref{eq:converge}, which indicates that accuracy of linearization model $\rho$ affects algorithm performance.

\section{Conclusion}\label{sec:conclusion}
We proposed a hierarchical distributed implementation of the primal-dual gradient algorithm to solve an OPF problem. The objective of OPF is to minimize the total cost over all the controllable DERs and a cost associated with the total network load, subject to voltage regulation constraints. By utilizing the information structure of tree/subtrees to reduce and distribute computational loads the proposed implementation is scalable to large multi-phase distribution networks. Performance of our design is analytically characterized and numerically corroborated. The significant improvement in convergence speed shows the great potential of the proposed method for grid optimization and control in real time.

\section*{Acknowledgments}
This work was authored in part by the National Renewable Energy Laboratory, operated by Alliance for Sustainable Energy, LLC, for the U.S. Department of Energy (DOE) under Contract No. DE-EE-0007998. Funding provided by U.S. Department of Energy Office of Energy Efficiency and Renewable Energy Solar Energy Technologies Office. The views expressed in the article do not necessarily represent the views of the DOE or the U.S. Government. The U.S. Government retains and the publisher, by accepting the article for publication, acknowledges that the U.S. Government retains a nonexclusive, paid-up, irrevocable, worldwide license to publish or reproduce the published form of this work, or allow others to do so, for U.S. Government purposes.

%%%%%%%%%%%%%%%%%%%%%%%%%%%%%%%%%%%%%%%%%%%%%
\bibliographystyle{IEEEtran}
\bibliography{biblio.bib}
%%%%%%%%%%%%%%%%%%%%%%%%%%%%%%%%%%%%%%%%%%%%%

\appendix

\subsection{Proof of Lemma~\ref{lem:mono}}
\begin{proof}
	Denote by $f(\bm{y})=\sum_{i\in\cN} C_i(p_i,q_i)+C_0(P_0(\bm{p}))$ and $\bm{\mu}^{\top}\bm{g}(\bm{y})=\underline{\bm{\mu}}^{\top}(\underline{\bm{v}}-\bm{v}(\bm{p},\bm{q}))+\overline{\bm{\mu}}^{\top}(\bm{v}(\bm{p},\bm{q})-\overline{\bm{v}})$ for simplicity. 
	%Based on assumptions, $f(\bm{x})$ is strongly convex in $\bm{y}$.
	Then $T(\bm{z})$ can be equivalently decomposed into the following operators:
	\begin{eqnarray}
	T(\bm{z})&\hspace{-3mm}=&\hspace{-3mm}\begin{bmatrix}\nabla_{\bm{y}} f(\bm{y}) \\ \nabla_{\bm{\mu}}\frac{\phi}{2}\|\bm{\mu}\|_2^2\end{bmatrix}+\begin{bmatrix}\nabla_{\bm{y}} \bm{\mu}^{\top}\bm{g}(\bm{y}) \\ -\nabla_{\bm{\mu}}\bm{\mu}^{\top}\bm{g}(\bm{y})\end{bmatrix}\nonumber\\
	&\hspace{-3mm}=&\hspace{-3mm}\begin{bmatrix}\nabla_{\bm{y}} f(\bm{y}) \\ \nabla_{\bm{\mu}}\frac{\phi}{2}\|\bm{\mu}\|_2^2\end{bmatrix}+\begin{bmatrix} 0 & 0 & -R^{\top} & R^{\top}\\ 0 & 0 & -X^{\top} & X^{\top}\\ R & X & 0 & 0 \\ -R & -X & 0 & 0 \end{bmatrix}\begin{bmatrix}\bm{p}\\\bm{q}\\\underline{\bm{\mu}}\\\overline{\bm{\mu}}\end{bmatrix}\nonumber\\
	&&+\text{Constant}.\nonumber
	\end{eqnarray}
	
	We can verify that the first operator $\begin{bmatrix}\nabla_{\bm{y}} f(\bm{y}) \\ \nabla_{\bm{\mu}}\frac{\phi}{2}\|\bm{\mu}\|_2^2\end{bmatrix}$ is strongly monotone since $f(\bm{y})$ and $\frac{\phi}{2}\|\bm{\mu}\|_2^2$ are strongly convex in $\bm{y}$ and $\bm{\mu}$, respectively. The second (linear) operator is monotone since
	\begin{eqnarray}
	\begin{bmatrix*}[c] 0  &\hspace{-2mm} 0 &\hspace{-2mm} -R^{\top} &\hspace{-2mm} R^{\top}\\ 0 &\hspace{-2mm} 0 & \hspace{-2mm} -X^{\top} &\hspace{-2mm} X^{\top}\\ R &\hspace{-2mm} X &\hspace{-2mm} 0 &\hspace{-2mm} 0 \\ -R &\hspace{-2mm} -X & \hspace{-2mm}0 &\hspace{-2mm} 0 \end{bmatrix*}+\begin{bmatrix*}[c] 0 & \hspace{-2mm}0 & \hspace{-2mm}-R^{\top} &\hspace{-2mm} R^{\top}\\ 0 &\hspace{-2mm} 0 & \hspace{-2mm}-X^{\top} & \hspace{-2mm}X^{\top}\\ R & \hspace{-2mm}X & \hspace{-2mm}0 & \hspace{-2mm}0 \\ -R & \hspace{-2mm}-X &\hspace{-2mm} 0 &\hspace{-2mm} 0 \end{bmatrix*}^{\top}\hspace{-1mm} \succeq 0. \nonumber
	\end{eqnarray}
	
	Therefore, $T(\bm{z})$ is a strongly monotone operator as the result of combining a strongly monotone operator and a monotone operator.
\end{proof}

\subsection{Proof of Theorem~\ref{the:converge}}
\begin{proof}
	We have
	\begin{eqnarray}
	\hspace{-3mm}&&\|\bm{z}(t+1)-\bm{z}^*\|_2^2\nonumber\\
	\hspace{-3mm}&\leq& \|\bm{z}(t)-\epsilon T(\bm{z}(t))-\bm{z}^*+\epsilon T(\bm{z}^*)\|_2^2\nonumber\\
	\hspace{-3mm}&=& \|\bm{z}(t)-\bm{z}^*\|_2^2+\|\epsilon T(\bm{z}(t))-\epsilon T(\bm{z}^*)\|_2^2\nonumber\\
	\hspace{-3mm}&&\hspace{3mm}-2\epsilon (\bm{z}(t)-\bm{z}^*)^{\top}(T(\bm{z}(t))- T(\bm{z}^*))\nonumber\\
	\hspace{-3mm}&\leq&(1+\epsilon^2L^2-2\epsilon M)\|\bm{z}(t)-\bm{z}^*\|_2^2\nonumber
	%+\epsilon^2L^2\|\bm{z}(t)-\bm{z}^*\|_2^2-2\epsilon M\|\bm{z}(t)-\bm{z}^*\|_2^2,\nonumber
	%\hspace{-3mm}&<& \Delta \|\bm{z}(t)-\bm{z}^*\|_2^2\nonumber,
	\end{eqnarray}
	where the first inequality comes from non-expansiveness of projection operator, and the second from \eqref{eq:strongmono} and \eqref{eq:lipschitz}. Then by Lemma~\ref{lem:sigmarelation} and condition \eqref{eq:converge0}, we have $0\leq 1+\epsilon^2L^2-2\epsilon M< 1$, i.e., $\|\bm{z}(t+1)-\bm{z}^*\|_2^2\leq \Delta \|\bm{z}(t)-\bm{z}^*\|_2^2$ for some constant $0<\Delta <1$.
	
	Therefore, dynamics \eqref{eq:mapk} converges to the unique saddle point of \eqref{eq:langr} exponentially fast.
\end{proof}

\end{document}